\documentclass[12pt]{amsart}
\usepackage{amssymb,amsmath,amsfonts,amscd,euscript,verbatim,}

\usepackage{fullpage}

\usepackage[backref=page]{hyperref}
\usepackage{todonotes}

\usepackage{xcolor}

\usepackage{bm}

\usepackage{tikz-cd}

\usepackage{relsize}

\numberwithin{equation}{section}
\newtheorem{thm}{Theorem}[section]
\newtheorem{prop}[thm]{Proposition}
\newtheorem{lem}[thm]{Lemma}
\newtheorem{cor}[thm]{Corollary}
\newtheorem{rem}[thm]{Remark}

\newtheorem*{thma}{Theorem A}
\newtheorem*{thmb}{Theorem B}
\newtheorem*{thmc}{Theorem C}
\newtheorem*{thmd}{Theorem D}
\newtheorem*{cor*}{Corollary}

\newcommand{\nc}{\newcommand}

\nc{\bA}{\mathbb A}
\nc{\bC}{\mathbb C}
\nc{\bc}{{\bf c}}
\nc{\bD}{\mathbb D}
\nc{\bd}{\mathbb d}
\nc{\bE}{\mathbb E}
\nc{\bG}{\mathbb G}
\nc{\bL}{\mathbb L}
\nc{\bN}{\mathbb N}
\nc{\bP}{\mathbb P}
\nc{\bQ}{\mathbb Q}
\nc{\bR}{\mathbb R}
\nc{\bu}{\mathbb u}
\nc{\bW}{\mathbb W}
\nc{\bZ}{\mathbb Z}

\nc{\cA}{\mathcal A}
\nc{\cB}{\mathcal B}
\nc{\cC}{\mathcal C}
\nc{\cD}{\mathcal D}
\nc{\cF}{\mathcal F}
\nc{\cI}{\mathcal I}
\nc{\cK}{\mathcal K}
\nc{\cL}{\mathcal L}
\nc{\cO}{\mathcal O}
\nc{\cP}{\mathcal P}
\nc{\cV}{\mathcal V}
\nc{\cW}{\mathcal W}

\nc{\al}{\alpha}
\nc{\be}{\beta}
\nc{\la}{\lambda}
\nc{\La}{\Lambda}
\nc{\ve}{\varepsilon}
\nc{\om}{\omega}

\nc{\gl}{\mathfrak{gl}}
\nc{\fsl}{\mathfrak{sl}}
\nc{\g}{\mathfrak{g}}
\nc{\gh}{\widehat\g}
\nc{\h}{\mathfrak{h}}
\nc{\fp}{{\mathfrak p}}
\nc{\fh}{{\mathfrak h}}
\nc{\fg}{{\mathfrak g}}
\nc{\fgh}{{\widehat{\mathfrak g}}}
\nc{\fb}{{\mathfrak b}}
\nc{\bh}{\widehat \fb}
\nc{\fn}{{\mathfrak n}}
\nc{\fQ}{\mathfrak{Q}}
\nc{\fI}{\mathfrak{I}}

\nc{\bi}{{\bold i}}

\nc{\Aut}{\mathrm{Aut}}
\nc{\ch}{{\mathop {\rm ch}}}
\nc{\tr}{{\mathop {\rm tr}\,}}
\nc{\id}{{\mathop {\rm id}}}
\nc{\ad}{{\mathop {\rm ad}}}
\nc{\gr}{\mathrm{gr}}
\nc{\Spec}{\operatorname{Spec}}
\nc{\Proj}{\operatorname{Proj}}
\nc{\Pic}{\operatorname{Pic}}
\nc{\Hom}{\mathrm{Hom}}
\nc{\Ann}{\mathrm{Ann}}
\nc{\wt}{\mathrm{wt}}
\nc{\hw}{\mathrm{hw}}
\nc{\rk}{\mathrm{rk}}
\nc{\Gr}{{\mathrm {Gr}}}
\nc{\Fl}{\mathrm{Fl}}
\nc{\spn}{\mathrm{span}}
\nc{\Rep}{\mathrm{Rep }}
\nc{\coh}{\mathrm{coh}}
\nc{\grmod}{\mathrm{grmod}}
\nc{\Ih}{\widehat I}

\nc{\eO}{\EuScript{O}}

\nc{\bra}{\langle}
\nc{\ket}{\rangle}
\nc{\pa}{\partial}
\nc{\ld}{\ldots}
\nc{\cd}{\cdots}
\nc{\hk}{\hookrightarrow}
\nc{\T}{\otimes}
\nc{\ov}{\overline}
\nc{\svee}{{\!\scriptscriptstyle\vee}}
\nc{\ula}{{\underline{\la}}}
\nc{\umu}{{\underline{\mu}}}
\nc{\conv}{{\widetilde \times}}
\nc{\lach}{{\la^\svee}}
\nc{\alch}{{\al^\svee}}
\nc{\omch}{{\omega^\svee}}

\DeclareMathOperator*{\bigstarlim}{\mathlarger{ \mathlarger { \mathlarger{ \mathlarger{\mathlarger{\star}}}}}}

\DeclareMathOperator*{\bigastlim}{\mathlarger{ \mathlarger { \mathlarger{ \mathlarger{\mathlarger{\ast}}}}}}

\nc{\GL}{\mathfrak{GL}}
\nc{\Tr}{{\mathop {\rm Tr}\,}}
\nc{\Id}{{\mathop {\rm Id}}}

\nc{\msl}{\mathfrak{sl}}
\nc{\mgl}{\mathfrak{gl}}

\nc{\U}{\mathrm U}

\nc{\Q}{\mathfrak Q}

\nc{\on}{\operatorname} \nc\ol{\overline} \nc\ul{\underline}

\nc{\BA}{{\mathbb{A}}} \nc{\BC}{{\mathbb{C}}} \nc{\BF}{{\mathbb{F}}}
\nc{\BD}{{\mathbb{D}}} \nc{\BG}{{\mathbb{G}}} \nc{\BQ}{{\mathbb{Q}}}
\nc{\BM}{{\mathbb{M}}} \nc{\BN}{{\mathbb{N}}} \nc{\BO}{{\mathbb{O}}}
\nc{\BP}{{\mathbb{P}}} \nc{\BR}{{\mathbb{R}}}
\nc{\BZ}{{\mathbb{Z}}} \nc{\BS}{{\mathbb{S}}} \nc{\BW}{{\mathbb{W}}}

\nc{\CA}{{\mathcal{A}}} \nc{\CL}{{\mathcal{L}}} \nc{\CV}{{\mathcal{V}}} \nc{\CW}{{\mathcal{W}}}
\nc{\CalD}{{\mathcal{D}}}

\nc{\sic}{{\on{sc}}}
\nc{\add}{{\on{add}}}

\begin{document}
	
	\title[Feigin--Loktev fusion product and coherent Satake category]
	{A geometric approach to Feigin--Loktev fusion product and cluster relations in coherent Satake category}
	
	\author{Ilya Dumanski}
	\address{Ilya Dumanski:\newline
		Department of Mathematics, MIT, Cambridge, MA 02139, USA,
		{\it and }\newline
		Department of Mathematics, National Research University Higher School of Economics, Russian Federation,
		Usacheva str. 6, 119048, Moscow.
	}
	\email{ilyadumnsk@gmail.com}

\begin{abstract}
	We propose a geometric realization of the Feigin--Loktev fusion product of graded cyclic modules over the current algebra. This allows us to compute it in several new cases. We also relate the Feigin--Loktev fusion product to the convolution of perverse coherent sheaves on the affine Grassmannian of the adjoint group. This relation allows us to establish the existence of exact triples, conjecturally corresponding to cluster relations in the Grothendieck ring of coherent Satake category.
\end{abstract}

\maketitle
	
\section*{Introduction}

\subsection{}
Let $\g$ be a semisimple Lie algebra over $\bC$ and $\g[t]$ the corresponding current algebra. In this paper, we deal with graded cyclic finite-dimensional $\g[t]$-modules.
In \cite{FeLo} the authors introduced a certain operation between such modules -- the fusion product. 

Since then, this operation and related concepts were found to be useful in the theory of quantum loop groups (\cite{CP}, \cite{Nao4} \cite{BCKV}), the theory of Macdonald polynomials (\cite{FM1}, \cite{BCSW}), the $X = M$ conjecture (\cite{AK}, \cite{DK1}, \cite{Nao4}), 
geometry of semi-infinite flag manifolds (\cite{Kat1}, \cite{Kat2}), Schur positivity (\cite{FH}, \cite{Nao3}, \cite{FMM}), toroidal algras (\cite{KL}, \cite{Fei}) and cluster categorifications (\cite{BCM}). We refer to \cite{BCKV} for a survey on this.

The definition of fusion product in \cite{FeLo} involves a choice of constants. Despite this definition being completely elementary, it is still not proved in general, that it in fact is independent of this choice. It is also not proved that this operation is associative. See \cite{CL}, \cite{FoLi2}, \cite{Nao1}, \cite{Nao4}, \cite{Ven}, \cite{CSVW}, \cite{VV}, \cite{KL}, \cite{CV}, \cite{Rav1}, \cite{Rav2} for special cases.

In this paper, we propose a geometric way of determining the fusion product of some modules. Note that studying representations of $\g[t]$ in the context of geometric representation theory is not a new idea. For instance, the relation of global Weyl modules to the geometry of semi-infinite flag varieties was studied in \cite{BF}, \cite{Kat1}, \cite{FM2}, \cite{DF}. In \cite{BFM}, the relation of fusion product with convolution of simple perverse coherent sheaves on the affine Grassmannian for $PSL_2$ was emphasized. Finally, in \cite{DFF}, \cite{HY}, the appearance of global Demzure modules over $\g[t]$ on the spaces of sections of the Schubert varieties in the Beilinson--Drinfeld Grassmannian was shown.
\subsection{}
Our approach is connected to the papers \cite{BFM} and \cite{DFF}. We show how one can determine the fusion product of modules $M_i$, which appear as spaces of sections of sheaves on certain varieties.
Namely, we have (we refer to Section \ref{section generalities} for all required definitions and to Section \ref{section main result} for precise statement and proof)

\begin{thma} \label{intro geometric fusion}
	(See Theorem \ref{geometric fusion} for precise statement)
	Let $X$ be a scheme over the affine space $\bA^{k}$, endowed with action of the Beilinson--Drinfeld current group scheme $G^{sc}_{\bA^{k}}(\cO)$. Assume that $\cF$ is a $G^{sc}_{\bA^{k}}(\cO)$-equivariant sheaf on $X$, flat over $\bA^{k}$. Suppose that for all $\bc = (c_1, \hdots, c_k) \in \bA^k$ with $c_i \neq c_j$ for $i \neq j$ one has an isomorpism of $G^{sc}_{\bA^{k}}(\cO) \vert_\bc \simeq G^{sc}(\cO)^{\times k}$ modules
	\begin{equation} \label{factorization of modules}
	    H^0(X \vert_{\bc}, \cF \vert_\bc)^* \simeq M_1(c_1) \T \hdots \T M_k(c_k).
	\end{equation}
	Suppose also that some technical conditions hold. Then the fusion product of modules $M_1, \hdots, M_k$ does not depend on the choice of constants and is isomorphic to $H^0(X \vert_{0}, \cF \vert_0)^*$.
\end{thma}

The equation \eqref{factorization of modules} is a variant of \textit{factorization property}, and examples of $X$, satisfying the conditions of this theorem, arise as closures of $G^{sc}_{\bA^{k}}(\cO)$-orbits on some ind-schemes, well-known in geometric representation theory. Probably the most well-known example of such an ind-scheme would be the Beilison--Drinfeld Grassmannian. However, most general results (to our knowledge) can be obtained by applying this theorem to the global convolution diagrams of $k$ affine Grassmannians and one affine flag variety (see Subsection \ref{subsection global convolution diagrams} for definition). Namely, from Theorem A we can deduce the following result about the fusion product of affine Demazure modules (see Subsection \ref{affine demazure modules} for definitions). 

\begin{thmb} (See Corollary \ref{fusion of general demazures} for precise statement) Let $\ell_1 \geq \hdots \geq \ell_k \geq \ell \geq 0$ be nonnegative integers, $\la_1^\svee, \hdots, \la_k^\svee$ be dominant coweights and $\mu$ be a dominant weight. Then
	the fusion product
	\[
	D(\ell_1, \ell_1 \la_1^\svee) \ast \hdots \ast D(\ell_k, \ell_k \la_k^\svee) \ast D(\ell, \mu)
	\]
	does not depend on the choice of constants and is isomorphic to a certain generalized Demazure module.
	
	
\end{thmb}

This result was mentioned in \cite[Conjecture 4.17]{Rav2} as a conjecture. Note also that it generalizes some known results about fusion products, previously proved by non-trivial manipulations with generators and relations. Say, if we let $\mu = 0$, $\ell_1 = \hdots = \ell_k$, we get \cite[Corollary 5]{FoLi2}.  If we let $\ell_1 = \hdots = \ell_{k + 1}$, we get \cite[Theorem 1]{VV}, earlier estabilished in some particular cases in \cite{Ven, CSVW}. If we let $\ell_1 = \hdots = \ell_{k}$ and $\mu$ is a multiple of the maximal root, we get \cite[Theorem 3.3]{Rav2}. If we let $\mu = 0$, all $\la_i^\svee$ are fundamental, and assume that $\g$ is of simply laced type, we get an alternative description of the fusion product from \cite[Theorem A]{Nao4}. See Remark \ref{relation to other fusion modules} for the geometric meaning of these results.

We hope that geometric techniques, similar to the one proposed in this paper, could be used to calculate the fusion product of other modules. See Remark \ref{fusion of quotient} and Remark \ref{fusion of evaluations} for possible directions.

\subsection{}
Another application of the geometric interpretation of Feigin--Loktev fusion product is the following.

Assume that $\g$ is of simply-laced type until the end of the Introduction.

Recall that in \cite{CW} the authors proved for $G = GL_n$ and conjectured for other groups that there is a cluster structure on the Grothendieck ring of the category of equivariant perverse coherent sheaves on the affine Grassmannian $\cP_{coh}^{G^{sc}(\cO)} (\Gr)$ (we do not encounter loop-rotation equivariance here for simplicity).

On the other hand, the cluster patterns of the category $\Rep \ U_q(\bold L \g)$ of finite-dimensional modules over the quantum loop group, were noticed by several authors (probably most well-known approach is due to Hernandez--Leclerc, see \cite{HL}; we do not appeal to it in the present paper). In particular, in \cite{Ked, DK2}, the cluster interpretation of Q-systems is suggested. The mutation matrix of the appearing cluster algebra is the same as conjectured in \cite{CW}. %


Despite this obvious similarity between the Grothendieck ring of the category $\cP_{coh}^{G^{sc}(\cO)} (\Gr)$ and the Q-systems ring, studied in \cite{Ked, DK2} (related to the Grothendieck ring of the category $\Rep \ U_q(\bold L \g)$), no precise statements about relation between these two categories seem to be present in the literature. We propose such a statement. Moreover, it helps to derive some new results about $\cP_{coh}^{G(\cO)} (\Gr)$, using the corresponding well-known results for $\Rep \ U_q(\bold L \g)$.

Q-systems are certain short exact sequences, whose terms are tensor products of Kirillov--Reshetikhin modules, which we denote by $KR_{i, \ell}$ (here $i$ is a vertex of Dynkin diagram and $\ell$ is a nonnegative integer). Thus, one can say that the role of cluster variables in the initial seed in the cluster algebra of \cite{Ked, DK2} is played by $KR_{i, \ell}$. The corresponding objects in $\cP_{coh}^{G^{sc}(\cO)} (\Gr)$ are simple real perverse coherent sheaves $\cP_{i, \ell}$ (see Section \ref{section coherent satake and quantum loop group} for definition).

For certain modules $M$ over $U_q(\bold L \g)$ their graded limit $M_{loc}$ is defined, which is a module over $\g[t]$ (see \cite{BCKV}). Then the following isomorphism of $\g[t]$-modules holds, which is a starting point of our connection:
\[
(KR_{i, \ell})_{loc} \simeq D(\ell, \ell \om_i^\svee) \simeq R\Gamma(\cP_{i, \ell})^*[-\frac{1}{2} \dim \Gr^{\om_i^\svee}],
\]
where $D(\ell, \ell \om_i^\svee)$ is an affine Demazure module
(note that $R\Gamma(\cP_{i, \ell})$ is concentrated in the cohomological degree $\frac{1}{2} \dim \Gr^{\om_i^\svee}$). We extend this isomorphism to ``products'' of these objects. Here ``product'' of perverse coherent sheaves is their convolution $\star$, ``product'' of $\g[t]$-modules is their Feigin--Loktev fusion product, and ``product'' of quantum loop group modules is just their tensor product (note that the relation between fusion product and convolution of perverse coherent sheaves on the affine Grassmannian for $PGL_2$ was stated in \cite[Proposition 8.5]{BFM}).

\begin{thmc}
	(Propositions \ref{sections of convolution of simple perverse sheaves}, \ref{graded limit of product of KR}, \ref{relation between perverse coherent and quantum group}) Suppose $\g$ is of simply-laced type. Let $i_1, \hdots, i_k$ be vertices of the Dynkin diagram and $\ell_1 \geq \hdots \geq \ell_k$ be nonnegative integers. Then
	\begin{multline*}
	(KR_{i_1, \ell_1} \T \hdots \T KR_{i_k, \ell_k})_{loc} \simeq D(\ell_1, \ell_1 \om_{i_1}^\svee) \ast \hdots \ast D(\ell_k, \ell_k \om_i^\svee) \simeq \\
	R \Gamma (\cP_{i_1, \ell_1} \star \hdots \star \cP_{i_k, \ell_k})^*[-\frac{1}{2} \dim \Gr^{\om_{i_1}^\svee} - \hdots - \frac{1}{2} \dim \Gr^{\om_{i_k}^\svee}]
	\end{multline*}
\end{thmc}


Further, let us point out how this connection can be exploited. In \cite{CW} the authors showed that the Grothendieck ring of $\cP_{coh}^{G^{sc}(\cO)} (\Gr)$ has a structure of cluster algebra for $G = GL_n$ and conjectured for other groups (\cite[Conjecture 1,10]{CW}). An important part of the proof for $GL_n$ is showing the existence of certain exact triangles (\cite[Proposition 2.17]{CW}), which correspond to cluster mutations of the initial cluster seed. We propose the analog of these exact triangles for all simply-laced types.

The idea of the proof is the following. We reduce the exactness of triple of sheaves to the exactness of global sections of these sheaves, shifted by large powers of the very ample (determinant) line bundle. The resulting $\g[t]$-modules turn out to be graded limits of Kirillov--Reshetikhin modules, which fit into Q-systems. Using \cite{CV}, this implies required exactness. So in short, we use the existence of Q-systems for modules over the loop quantum group to deduce their analogs for perverse coherent sheaves.

\begin{thmd}
	(See Theorem \ref{exact triples of perverse coherent} for precise statement) There are certain exact sequences of perverse coherent sheaves on the affine Grasmannian. They are analogs of Q-systems for modules over $U_q(\bold L \g)$ and they correspond to cluster relations in the Grothendieck ring of $\cP_{coh}^{G^{sc}(\cO)} (\Gr)$.
\end{thmd}

Note that in \cite[Subsection 6.4]{CW} it is already stated, that taking global sections reduces these short exact sequences of sheaves to the graded limits of Q-systems. However, the authors only take into account the $\g$-action on the global sections, which seems to be too weak structure to go in the opposite direction (deduce Theorem D from the existence of  Q-systems). Taking into account $\g[t]$-action resolves this problem (note also that in general the graded limit of $KR_{i, \ell}$ is not isomorphic to $V(\ell \om_i)$ outside of type A).

We expect that a similar proof can be given in non-simply laced case as well.

\subsection{}
The paper is organized as follows.

In Section 1 we fix the notation, define all required algebraic and geometric objects and recall their properties to be used.

In Section 2 we present our main result, which gives a geometric interpretation of the Feigin--Loktev fusion product (Theorem \ref{geometric fusion}). We apply it to the global convolution diagram and compute the fusion product in the new case in Corollary \ref{fusion of general demazures}.

In Section 3 we prove a relation between the Feigin--Loktev fusion product and convolution of perverse coherent sheaves (Proposition \ref{sections of convolution of simple perverse sheaves}). That allows to state a connection between categories of equivariant perverse coherent sheaves on the affine Grassmannian and modules over the quantum loop group (Proposition \ref{relation between perverse coherent and quantum group}). Using it, we deduce the existence of short exact sequences of perverse coherent sheaves (Theorem \ref{exact triples of perverse coherent}).

\section*{Acknowledgments}
This paper owes its existence to Michael Finkelberg. Most of the ideas presented here appeared during conversations with him. It is also my pleasure to thank Evgeny Feigin, Sabin Cautis, Harold Williams, Roman Travkin, Roman Bezrukavnikov, and Ivan Karpov for useful discussions. I am also grateful to the anonymous referee for valuable comments and suggestions.

I am also indebted to Michael Finkelberg and Alexander Popkovi\v{c} for reading the draft of this paper and correcting a lot of mistakes.	
	
\section{Generalities}\label{section generalities}

\subsection{Notations}
We fix $\g$ to be a semisimple Lie algebra over $\bC$. $I$ denotes the set of its Dynkin diagram vertices.
Let $Q$, $Q^\vee$, $P$, $P^\vee$ be its lattices of roots, coroots, weights, and coweights correspondingly. By $P^+$ ($P^{\vee +}$) we denote the cone of dominant weights (coweights). 
The invariant form on $Q^\vee$, such that the norm of short coroot equals 2, yields the inclusion $Q^\vee \hookrightarrow P$, which extends to the inclusion of lattices $\iota: P^\vee \hookrightarrow P$, which is an isomorphism iff $\g$ is of simply-laced type. By the partial relation $<$ on $P, P^\vee$ we mean the dominance order. For $\la \in P^+$ we denote by $V(\la)$ the corresponding irreducible $\g$-module.
Let $\rho$ be the half sum of positive roots, or, equivalently, the sum of all fundamental weights.

Consider the affine Kac-Moody Lie algebra $\gh$ \footnote{We use this place to fix a mistake from our previous paper \cite{DFF}: in Section 1.3 we claimed that affine dominant weights of level 1 are in bijeciton with $P^\vee/Q^\vee$. This is false: for example, in type $G_2$, there are two level 1 dominant weights ($\Lambda_0$ and $\Lambda_0 + \om$, where $\om$ is the short fundamental weight), but $|P^\vee/Q^\vee| = 1$. This is not used anywhere later in that paper, and all results of that paper hold true.}; let $\fI \subset \g[t] \subset \gh$ be the Iwahori subalgebra. $\Ih$ denotes the set of affine Dynkin diagram vertices. Let $\widetilde{W} = W \ltimes Q^\vee$ be the affine Weyl group and 
$\widehat{W} = W \ltimes P^\vee$ be the extended affine Weyl group. For a coweight $\la^\svee$ we denote by $\tau_{\la^\svee}$ the corresponding (translation) element of $\widehat{W}$. 
By $w_0$ we denote the longest element of $W$.  

We denote by $\widehat{P}$ ($\widehat{P}^+$) the set of affine weights (dominant weights). $\La_0$ denotes the vacuum weight.

\subsection{Feigin--Loktev fusion product} \label{FL fusion profuct}
In \cite{FeLo} the following operation was introduced. Let $M_1, \hdots, M_k$ be cyclic graded finite-dimensional modules over $\g[t]$. We assume that each $M_i$ has a unique up to scaling highest $\g$-weight vector $m_i$, and this vector is cyclic. For $c_i \in \bC$ we define $M_i(c_i)$ to be a $\g[t]$-module, which is obtained as a composition of the twist $\g[t] \rightarrow \g[t]$, $xt^s \mapsto x(t + c_i)^s$ (here $x \in \g$), with the action of $\g[t]$ on $M_i$ (so $M_i(c_i)$ is isomorphic to $M_i$ as a vector space, but $xt^s \in \g[t]$ acts as $x(t + c_i)^s$). $M_i(c_i)$ is cyclic (with the same cyclic vector $m_i$), but no longer graded. By \cite[Proposition 1.4]{FeLo}, for $c_i \neq c_j$ the tensor product
\[
M_1(c_1) \T \hdots \T M_k(c_k)
\]
is cyclic with the cyclic vector $m_1 \T \hdots \T m_k$.
The grading on $U(\g[t])$ by $t$-degree induces the filtration on this tensor product:
\[
F_{\leq n} ( M_1(c_1) \T \hdots \T M_k(c_k) ) = U(\g[t])_{\leq n} (m_1 \T \hdots \T m_k).
\]
The Feigin--Loktev fusion product is defined as the associated graded module:
\[
M_1 \ast \hdots \ast M_k = \gr_{F}  (M_1(c_1) \T \hdots \T M_k(c_k)).
\]

It was conjectured in \cite{FeLo} that the resulting module does not depend on the choice of parameters $c_i$, and that the operation of fusion product is associative.

It is beyond our current possibilities to prove this conjecture for arbitrary modules $M_i$. But one suitable family of modules for which
some results can be obtained are the affine Demazure modules.

\subsection{Affine Demazure modules}  \label{affine demazure modules}
We refer to \cite{BCKV, CV, Nao2, Rav2} for details on (generalized) affine $\g[t]$-invariant Demazure modules.

To an integrable dominant affine weight $\La \in \widehat{P}^+$ we associate the irreducible highest weight $\gh$-representation $V(\La)$.
Let $\La_1, \hdots, \La_k \in \widehat{P}^+$, and $\xi_1, \hdots, \xi_k \in \widehat{P}$, such that $\xi_i \in \widetilde{W} \La_i$ for any $i$. Since dimensions of weight spaces of $V(\La)$ are invariant under the action of $\widetilde{W}$, for any $i$ there is a unique (up to scaling) vector $v_{\xi_i} \in V(\La_i)$ of weight $\xi_i$. The generalized affine Demazure module is defined as 
\[
\bold D(\xi_1, \hdots, \xi_k) := U(\fI). (v_{\xi_1} \T \hdots \T v_{\xi_k}) \subset V(\La_1) \T \hdots V(\La_k).
\]
By definition, it is a module over $\fI$, but in fact, if $\langle \xi_i, \alch \rangle \leq 0$ for any $i$ and positive coroot $\alch$, it is $\g[t]$-invariant (recall that $\fI \subset \g[t]$).

The most well-known case is $k = 1$. In this case the module $\bold D(w \La)$ for $w \in \widehat{W}, \La \in \widehat{P}^+$ (sometimes we will denote it by $V_w (\La)$) is called the affine Demazure module (note that if $w$ is an element of $\widehat{W}$ rather than of $\widetilde{W}$, the vector of weight $w\La$ may lie not in $V(\La)$ but in some other irreducible highest weight $\gh$-module).

The $\g[t]$-invariant affine Demazure modules are naturally parametrized by the set $\bZ_{\ge 0} \times P^+$.
Namely, if for $w \in \widehat{W}, \La \in \widehat{P}^+, \mu \in P^+$ one has $w  \La = \ell \La_0 + w_0\mu$, then  $\bold D(w \La)$ is $\g[t]$-invariant. We denote such a module by $D(\ell, \mu)$. 
$D(\ell, \mu)$ is naturally a graded cyclic highest weight $\g[t]$-module with cyclic vector of weight $\mu$. Note that in this case $w$ can me presented in a form $\tau_{w_0 \nu^\svee}v$ with $\nu^\svee \in \widehat{P}^{\vee +}$, $v \in W$. 

The modules of the form $D(\ell, \ell \iota(\la^\svee))$, where $\la^\svee$ is a dominant coweight, are particularly well-studied (see, for instance, \cite{FoLi1, FoLi2}). For them the corresponding (extended) affine Weyl group element is $\tau_{w_0\la^\svee}$, and the affine dominant weight is $\ell \La_0$, 
so $D(\ell, \ell\iota(\la^\svee)) \simeq \bold D(\tau_{w_0\la^\svee}(\ell \La_0) )$. 
We omit the symbol $\iota$ in their notation and simply write $D(\ell, \ell \la^\svee)$ (note that the authors of \cite{FoLi2}, and some other papers use the notation $D(\ell, \la^\svee)$ for the same modules).

\subsection{Affine Grassmannian and affine flag variety}
Fix the notations $\cK = \bC((t)), \cO = \bC[[t]]$.
Let $G^{sc}$ be the simply-connected group, corresponding to $\g$, and $G = G^{ad}$ be the adjoint group.

We denote by $\Gr$ the affine Grassmannian of the group $G = G^{ad}$.
It has the left action by the group $G^{sc}(\cO)$. All sheaves we consider are considered as $G^{sc}(\cO)$-equivariant sheaves. Global sections of such sheaves are naturally $G^{sc}(\cO)$-representations (although we usually consider them as $\g[t]$-representations). Informally speaking, consideration of $G^{sc}(\cO)$-equivariant sheaves on $\Gr_{G^{ad}}$ allows us to get as many $\g[t]$-representations geometrically as possible. See \cite[Section 3]{FT} for consideration of $G^{sc}(\cO)$-equivariant sheaves on $\Gr_{G^{ad}}$ (in non-simply-laced types the twisted version is considered there, which we do not need here). Note also that it agrees with \cite[Conjecture 1.10]{CW}, which we are going to address in Section 3.

There is a very ample $G^{sc}(\cO)$-equivariant line bundle, called the determinant line bundle, and denoted by $\cL$. For any $\la^\svee \in P^\vee$, there is a point $t^{\la^\svee} \in \Gr$. All $G(\cO)$-orbits on $\Gr$ are of the form $G(\cO) t^{\la^\svee}$ for $\la^{\svee} \in P^{\vee +}$, and we denote $\Gr^{\la^\svee} = G(\cO) t^{ \la^\svee}$. The closures of these orbits are called the spherical affine Schubert varieties. Their dimensions are given by: $\dim \ov \Gr^{\la^\svee} = \langle 2 \rho, \la^\svee \rangle$. 
The Borel--Weil--Bott-type theorem claims that for $\ell \geq 0$ there are isomorphisms of $\g[t]$-modules (see \cite[Theorem 8.2.2 (a),(c)]{Kum}):
\begin{equation} \label{cohomology of spherical schubert variety}
\begin{aligned}
H^0(\ov \Gr^{\la^\svee}, \cL^{ \ell}) &\simeq D(\ell, \ell \la^\svee)^*, \\ 
H^{> 0}(\ov \Gr^{\la^\svee}, \cL^{ \ell}) &\simeq 0.
\end{aligned}
\end{equation}

Let $\Fl$ be the affine flag variety of $G$. It has a line bundle, corresponding to any affine weight $\La$, denote this bundle by $\cL_\La$. For any $w \in \widehat{W}$ there is a point $t^{w} \in \Fl$. These points represent different orbits under the (left) action of the Iwahori group $\cI \subset G(\cO)$, 
we denote these orbits by $\Fl^w$.
The closures of these orbits are called the affine Schubert varieties. 
For dominant $\La$ there is an isomorphism of $\fI$-modules (we refer to \cite[Theorem 8.2.2 (a),(c)]{Kum} again):
$H^0(\ov \Fl^w, \cL_\La) \simeq \bold D(w\La)^*$.
If $w$ has the form $\tau_{w_0 \nu^\svee} v$ with $\nu^\svee \in \widehat{P}^{\vee +}$, $v \in W$, then  $\ov \Fl^w$ is $G(\cO)$-invariant, in which case $\bold D (w\La)$ is also $\g[t]$-invariant, and it is an isomorphism of $\g[t]$-modules. Namely, if $\La \in \widehat{P}^+$ is dominant, $w\La = \ell\La_0 + w_0\mu$, and $\mu \in P^+$ is dominant, then one has isomorphisms of $\g[t]$-modules:
\begin{equation} \label{cohomology of schubert variety}
\begin{aligned}
H^0(\ov \Fl^w, \cL_\La) &\simeq D(\ell, \mu)^*, \\
H^{> 0}(\ov \Fl^w, \cL_\La) &\simeq 0.
\end{aligned}
\end{equation}

One also has a natural projection $\Fl \rightarrow \Gr$, which a locally trivial fibration with fiber $G/B$. It restricts to the fibration
\begin{equation} \label{fibration of flags over gr}
\ov \Fl^{\tau_{w_0 \lach} w_0} \rightarrow \ov \Gr^\lach
\end{equation}
with the same fiber $G/B$; here $\lach \in P^{\vee +}$ is dominant.

\subsection{Convolution diagrams}
The convolution diagram of affine Grassmannians is defined as:
\[
\Gr^{\conv k} = \underbrace{G(\cK) \times^{G(\cO)} \hdots \times^{G(\cO)} G(\cK) \times^{G(\cO)} \Gr}_{k},
\]
where $X \times^{G(\cO)} Y$ means $X \times Y$ modulo the relation $(xg, y) \sim (x, gy)$ for $g \in G(\cO)$ (we mean natural right and left actions of $G(\cO)$ here). 

Let $p_\Gr: G(\cK) \rightarrow \Gr$ and $p_\Fl: G(\cK) \rightarrow \Fl$ be the natural projections. We also define
\[
\ov \Gr^{\la^\svee_1} \conv \hdots \conv \ov \Gr^{\la^\svee_k} = p_{\Gr}^{-1}(\ov \Gr^{\la^\svee_1}) \times^{G(\cO)} \hdots \times^{G(\cO)} p_{\Gr}^{-1}(\ov \Gr^{\la^\svee_{k - 1}}) \times^{G(\cO)} \ov \Gr^{\la^\svee_k}.
\]

In what follows, we also need the convolution diagram of affine Grasmannians and an affine flag variety:
\[
\Gr^{\conv k} \conv \Fl = \underbrace{G(\cK) \times^{G(\cO)} \hdots \times^{G(\cO)} G(\cK)}_{k} \times^{G(\cO)} \Fl.
\]
If $w \in \widehat{W}$ is such that $\ov \Fl^w$ is $G(\cO)$-invariant, we can define
\[
\ov \Gr^{\la^\svee_1} \conv \hdots \conv \ov \Gr^{\la^\svee_k} \conv \ov \Fl^w = p_{\Gr}^{-1}(\ov \Gr^{\la^\svee_1}) \times^{G(\cO)} \hdots \times^{G(\cO)} p_{\Gr}^{-1}(\ov \Gr^{\la^\svee_{k}}) \times^{G(\cO)} \ov \Fl^w.
\]

For any $1 \leq i \leq k$ we can define the (partial) convolution morphism 
\begin{equation} \label{partial convolution morphisms}
\begin{aligned}
\theta_i: \Gr^{\conv k} \conv \Fl &\rightarrow \Gr \\
(g_1, \hdots, g_k, x) &\mapsto p_\Gr(g_1\cdot \hdots \cdot g_i),
\end{aligned}
\end{equation}
and the morphism $\theta_{k+1}: \Gr^{\conv k} \conv \Fl \rightarrow \Fl$ is defined similarly.
The morphisms $\theta_i$ all together define an isomorphism:
\begin{equation} \label{convolution isomorphic to product}
\Theta = (\theta_1, \hdots, \theta_{k + 1}): \Gr^{\conv k} \conv \Fl \xrightarrow{\sim} \underbrace{\Gr \times \hdots \Gr}_k \times \Fl.
\end{equation}

For $(\ell_1, \hdots, \ell_k, \La) \in \bZ^k \oplus \widehat{P}$ we denote by $\cL^{(\ell_1, \hdots, \ell_k, \La)}$ 
the following line bundle on $\Gr^{\conv k} \conv \Fl$: 
\begin{equation} \label{definition of line bindle L}
\cL^{(\ell_1, \hdots, \ell_k, \La)} := \Theta^*( \cL^{\ell_1 - \ell_2} \boxtimes \hdots \boxtimes \cL^{\ell_k - \ell} \boxtimes \cL_\La),
\end{equation}
where $\La$ has level $\ell$.



\begin{prop} \label{sections of convolution diagram}
Let $\la^\svee_1, \hdots, \la_k^\svee \in P^{\vee +}$, $\La \in \widehat{P}^+$, $w \in \widehat{W}$, such that $w \La = \ell \La_0 + w_0 \mu$, $\mu \in P^+$. Let $\ell_1 \geq \hdots \geq \ell_k \geq \ell \geq 0$. Then one has an isomorphism of $\g[t]$-modules
\begin{multline*}
H^0(\ov \Gr^{\la^\svee_1} \conv \hdots \conv \ov \Gr^{\la^\svee_k} \conv \ov \Fl^w, \cL^{(\ell_1, \hdots, \ell_k, \La)})^* \\
\simeq \bold D \Bigl( \tau_{w_0 \la_1^\svee} (\ell_1 - \ell_2)\La_0, \tau_{w_0(\la_1^\svee + \la_2^\svee)} (\ell_2 - \ell_3)\La_0, \hdots, \tau_{w_0(\la_1^\svee + \hdots + \la_k^\svee)}(\ell_k - \ell) \La_0, \tau_{w_0(\la_1^\svee + \hdots + \la_k^\svee)} w \La \Bigr),
\end{multline*}
and 
\[
H^{> 0}(\ov \Gr^{\la^\svee_1} \conv \hdots \conv \ov \Gr^{\la^\svee_k} \conv \ov \Fl^w, \cL^{(\ell_1, \hdots, \ell_k, \La)}) = 0.
\]
\end{prop}

\begin{proof}
Our main tool is \cite[Theorem 6]{LLM}, which is a very similar result, and we reduce our claim to this theorem. Note that this theorem is stated in notations of the Lie group of finite type, but as claimed in Section 1.1 of \textit{loc. cit.}, the results are valid and straightforward to generalize for any Kac--Moody type, including the required for us affine case.

Recall the projection \eqref{fibration of flags over gr}. Evidently, we also have a locally trivial fibration
\begin{equation} \label{convolution of flags to convolution of grassmannians}
p_{\Gr}^{-1}(\ov \Gr^{\la^\svee_1}) \times^\cI \hdots \times^\cI p_{\Gr}^{-1}(\ov \Gr^{\la^\svee_{k}}) \times^\cI \ov \Fl^w \rightarrow \ov \Gr^{\la^\svee_1} \conv \hdots \conv \ov \Gr^{\la^\svee_k} \conv \ov \Fl^w
\end{equation}
with fiber $(G/B)^k$. We denote the left-hand side of it by $\ov \Fl^{\tau_{w_0 \lambda^\svee_1}w_0} \conv \hdots \conv \ov \Fl^{\tau_{w_0 \lambda^\svee_k}w_0} \conv \ov \Fl^w$. Now, let
\[
\bold i = (i^1_1, \hdots, i^1_{\bm{\ell}_1}, \hdots, i^k_1, \hdots, i^k_{\bm{\ell}_k}, i_1, \hdots, i_{\bm{\ell}}),
\]
where all $i^n_j$ are in $\Ih$, and $(i^n_1, \hdots, i^n_{\bm{\ell}_1})$ is a reduced expression in the affine Weyl group of $\tau_{w_0 \lambda^\svee_n}w_0$ for any $1 \leq n \leq k$, and $(i_1, \hdots, i_{\bm{\ell}})$ is a reduced expression for $w$.

Consider the Bott--Samelson variety $Z_{\bi}$ (in the spirit of our notation, it might have been denoted $\ov \Fl^{s_{i_1^1}} \conv \hdots \conv \ov \Fl^{s_{i_{\bm \ell}}}$; see \cite[Chapters VII, VIII]{Kum} for detailed treatment of these varieties). 
There is a partial convolution morhism $Z_{\bm i} \rightarrow \ov \Fl^{\tau_{w_0 \lambda^\svee_1}w_0} \conv \hdots \conv \ov \Fl^{\tau_{w_0 \lambda^\svee_k}w_0} \conv \ov \Fl^w$.
Composing with \eqref{convolution of flags to convolution of grassmannians}, we have a surjective morphism $\phi$:
\[
Z_{\bi} \stackrel{\phi}{\rightarrow} \ov \Gr^{\la^\svee_1} \conv \hdots \conv \ov \Gr^{\la^\svee_k} \conv \ov \Fl^w.
\]
Now, clearly, the line bundle $\phi^* \cL^{(\ell_1, \hdots, \ell_k, \La)}$ on $Z_{\bi}$ has a form $\cL_{\bi, \bold m}$ for certain $\bold m$ in notation of \cite[1.4]{LLM}. Moreover, if $\cL^{(\ell_1, \hdots, \ell_k, \La)}$ is very ample, 
then $\ov \Gr^{\la^\svee_1} \conv \hdots \conv \ov \Gr^{\la^\svee_k} \conv \ov \Fl^w$ coincides with image of the natural map $Z_\bi \rightarrow \bP(H^0(Z_\bi, \phi^*\cL^{(\ell_1, \hdots, \ell_k, \La)})^*)$, denoted by $Z_{\bi, \bold m}$ in \cite{LLM}. 
Hence, in this case, the desired isomorphisms (at the level of $\fI$-modules) follow from \cite[Theorem~6]{LLM}. In the general case, the morphism $Z_\bi \rightarrow Z_{\bi, \bold m}$ factors through $\ov \Gr^{\la^\svee_1} \conv \hdots \conv \ov \Gr^{\la^\svee_k} \conv \ov \Fl^w$, and we argue as follows.

We have morphisms
\begin{multline*}
Z_{\bi} \stackrel{\phi}{\rightarrow} \ov \Gr^{\la^\svee_1} \conv \hdots \conv \ov \Gr^{\la^\svee_k} \conv \ov \Fl^w \rightarrow Z_{\bi, \bold m} \rightarrow \\ 
\bP(\bold D \Bigl( \tau_{w_0 \la_1^\svee} (\ell_1 - \ell_2)\La_0, \hdots, \tau_{w_0(\la_1^\svee + \hdots + \la_k^\svee)}(\ell_k - \ell) \La_0, \tau_{w_0(\la_1^\svee + \hdots + \la_k^\svee)} w \La \Bigr)),
\end{multline*}
where first two morphisms are surjective, as evident from the construction above, and the last morphism is constructed in \cite[Section 4.1]{LLM}, and it was shown there, that its image linearly spans the whole projective space. From this, we have the injective homomorphisms of $\fI$-modules
\begin{multline*}
\bold D \Bigl( \tau_{w_0 \la_1^\svee} (\ell_1 - \ell_2)\La_0, \hdots, \tau_{w_0(\la_1^\svee + \hdots + \la_k^\svee)}(\ell_k - \ell) \La_0, \tau_{w_0(\la_1^\svee + \hdots + \la_k^\svee)} w \La \Bigr)^* \stackrel{\al}{\hookrightarrow} \\
H^0(\ov \Gr^{\la^\svee_1} \conv \hdots \conv \ov \Gr^{\la^\svee_k} \conv \ov \Fl^w, \cL^{(\ell_1, \hdots, \ell_k, \La)}) \hookrightarrow H^0(Z_\bi, \phi^* \cL^{(\ell_1, \hdots, \ell_k, \La)}),
\end{multline*}
whose composition is proved to be bijective in \cite[Section 4.2]{LLM}. It follows, that $\al$ is an isomorphism. Moreover, it is evident from the construction of \cite[Section 4.1]{LLM} that the morphism
\[
\ov \Gr^{\la^\svee_1} \conv \hdots \conv \ov \Gr^{\la^\svee_k} \conv \ov \Fl^w \rightarrow 
\bP(\bold D \Bigl( \tau_{w_0 \la_1^\svee} (\ell_1 - \ell_2)\La_0, \hdots, \tau_{w_0(\la_1^\svee + \hdots + \la_k^\svee)}(\ell_k - \ell) \La_0, \tau_{w_0(\la_1^\svee + \hdots + \la_k^\svee)} w \La \Bigr))
\]
is in fact $G(\cO)$-equivariant, not just $\cI$-equivariant; hence, its dual morphism of sections $\al$ is also an isomorphism of $\g[t]$-modules. The first claim of the Proposition is proved.

For the part about the higher cohomologies, the same argument as for $Z_{\bold i, \bold m}$ in \cite[Section 4.2]{LLM} works. Namely, the morphism $Z_{\bi} \stackrel{\phi}{\rightarrow} \ov \Gr^{\la^\svee_1} \conv \hdots \conv \ov \Gr^{\la^\svee_k} \conv \ov \Fl^w$ satisfies the assumptions of the Kempf's lemma \cite{Kem}, hence normality of $\ov \Gr^{\la^\svee_1} \conv \hdots \conv \ov \Gr^{\la^\svee_k} \conv \ov \Fl^w$ implies that $H^{> 0} (\ov \Gr^{\la^\svee_1} \conv \hdots \conv \ov \Gr^{\la^\svee_k} \conv \ov \Fl^w, \cL^{(\ell_1, \hdots, \ell_k, \La)}) = H^{> 0} (Z_\bi, \phi^* \cL^{(\ell_1, \hdots, \ell_k, \La)}) = 0$, as required.
\end{proof}

In the same way, one can generalize \cite[Theorem 6]{LLM} for convolution of Schubert varieties of arbitrary partial flag variety of a Kac--Moody group. We will not need it here.

\subsection{Global convolution diagrams}  \label{subsection global convolution diagrams}
An important feature of the convolution diagrams is that they admit ``globalization''.  The global convolution diagram of the affine Grassmannians is particularly well-known, because it may be used to define the commutativity constraint in the geometric Satake equivalence.


We need a variant of this construction, the global convolution diagram of $k$ affine Grassmannians and one affine flag variety.

We define $\Gr_{\bA^1}^{\conv k} \conv \Fl_{\bA^1}$ as the moduli space of collections
\[
\{ (c_1, \hdots, c_{k + 1}), \mathcal{P}_1, \hdots, \cP_{k + 1}, \alpha_1, \hdots,  \alpha_{k + 1}, \epsilon \},
\]
where $(c_1, \hdots, c_{k + 1}) \in \bA^{k + 1}$, $\cP_1, \hdots , \cP_{k + 1}$ are $G$-torsors on $\bA^1$, $\alpha_{1}$ is a trivialization 
\[
\cP_{1} \vert_{\bA^1 \setminus \{ c_{1} \}} \xrightarrow{\sim} \cP^{triv} \vert_{\bA^1 \setminus \{ c_{1} \}},
\]
$\alpha_i$ for $2 \leq i \leq k + 1$  is an isomorphism $\cP_i \vert_{\bA^1 \setminus \{ c_i \}} \xrightarrow{\sim} \cP_{i - 1} \vert_{\bA^1 \setminus \{ c_{i} \}}$, and $\epsilon$ is a $B$-reduction of $\cP_{k + 1}$ at the point $c_{k + 1}$.

This space coincides with the space denoted $\widehat{\bf{Gr}}_{\id}^{P, \emptyset}$ 
with $P = \{1, \hdots, k \}$ in \cite[Section~4.4]{AR1} (the only difference is, $c_{k + 1}$ is fixed to be 0 in \cite{AR1}, but ``allowing $c_{k + 1}$ to move'' does not change any argument about this space). We refer to this paper and references therein (mainly \cite{AR2}) for details on this (and more general) spaces.

$\Gr_{\bA^1}^{\conv k} \conv \Fl_{\bA^1}$ is an ind-scheme, equipped with the ind-projective map $\pi: \Gr_{\bA^1}^{\conv k} \conv \Fl_{\bA^1} \rightarrow \bA^{k + 1}$. The fibers of this map are:
\begin{equation} \label{fibers of global convolution diagram}
\begin{aligned}
\Gr_{\bA^1}^{\conv k} \conv \Fl_{\bA^1} \vert_{\bc} &= \Gr^{\times k} \times \Fl, \text{ if } \bc = (c_1, \hdots c_{k + 1}), c_i \neq c_j; \\
\Gr_{\bA^1}^{\conv k} \conv \Fl_{\bA^1} \vert_{0} &= \Gr^{\conv k} \conv \Fl, \
\end{aligned}
\end{equation}
and for general $\bc$, the fiber depends on which of coordinates of $\bc$ are equal.

\subsection{Beilinson--Drinfeld current group scheme} One of the natural group schemes, acting on the global convolution diagram, is the Beilinson--Drinfeld current group scheme. It is an ind-affine ind-group scheme $G_{\bA^{k + 1}}(\cO)$ over $\bA^{k + 1}$, whose fiber over the point $\bc = (c_1, \hdots, c_{k + 1})$ is isomorphic to the inverse limit
\begin{equation} \label{fibers of BD group}
G_{\bA^{k + 1}}(\cO) \vert_\bc = \lim\limits_{\infty \leftarrow i} G(\bC[t] / P(t)^i),
\end{equation}
where $P(t) = (t - c_1) \hdots (t - c_{k + 1})$. We refer to \cite[Section 3.1]{Zhu} for the full definition.

It is clear from \eqref{fibers of BD group} that  $G_{\bA^{k + 1}}(\cO)\vert_\bc$ is isomorphic to $G(\cO)^{\times m}$, where $m$ is the number of nonequal numbers among $c_1, \hdots, c_{k + 1}$.

\subsection{Schubert varieties in the global convolution diagrams}
We refer to \cite[Sections 4.4, 4.5]{HY} for treatment of the material of this subsection, where the case of $\Gr_{\bA^1}^{\conv k}$ was 
considered. All the statements we make for $\Gr_{\bA^1}^{\conv k} \conv \Fl_{\bA^1}$ are proved similarly to \cite{HY}.

The group scheme $G_{\bA^{k + 1}}(\cO)$ acts naturally on $\Gr_{\bA^1}^{\conv k} \conv \Fl_{\bA^1}$. The orbits of this action are parametrized by the tuples $(\la_1^\svee, \hdots, \la_k^\svee, w)$, where $\la_1^\svee, \hdots, \la_k^\svee$ are dominant coweights, and $w$ is an affine Weyl group element, having the form $\tau_{w_0\nu^\svee} v$ with $\nu^\svee \in P^{\vee +}, v \in W$ (so $\Fl^w$ is $G(\cO)$-invariant). We denote the corresponding orbit by $\Gr_{\bA^1}^{\la_1^\svee} \conv \hdots \conv \Gr_{\bA^1}^{\la_k^\svee} \conv \Fl_{\bA^1}^w$. 
The fibers of the closure $\ov \Gr_{\bA^1}^{\la_1^\svee} \conv \hdots \conv \ov \Gr_{\bA^1}^{\la_k^\svee} \conv \ov \Fl_{\bA^1}^w$ over $\bA^{k + 1}$ are easy to describe. Namely, one has
\begin{equation} \label{fibers of convolution schubert variety}
\begin{aligned}  
\ov \Gr_{\bA^1}^{\la_1^\svee} \conv \hdots \conv \ov \Gr_{\bA^1}^{\la_k^\svee} \conv \ov \Fl_{\bA^1}^w \vert_{\bc} &= \ov \Gr^{\la_1^\svee} \times \hdots \times \ov \Gr^{\la_k^\svee} \times \ov \Fl^w, \text{ if } \bc = (c_1, \hdots c_{k + 1}), c_i \neq c_j; \\
\ov \Gr_{\bA^1}^{\la_1^\svee} \conv \hdots \conv \ov \Gr_{\bA^1}^{\la_k^\svee} \conv \ov \Fl_{\bA^1}^w \vert_{0} &= \ov \Gr^{\la_1^\svee} \conv \hdots \conv \ov \Gr^{\la_k^\svee} \conv \ov \Fl^w, \
\end{aligned}
\end{equation}
and for general $\bc$, the fiber depends on which of coordinates of $\bc$ are equal.

\begin{prop} \label{flatness of global convolution}
	The structural morphism 
	\[
	\pi_{(\ula^\svee, w)}: \ov \Gr_{\bA^1}^{\la_1^\svee} \conv \hdots \conv \ov \Gr_{\bA^1}^{\la_k^\svee} \conv \ov \Fl_{\bA^1}^w \rightarrow \bA^{k + 1}
	\]
	is flat.
\end{prop}

\begin{proof}

	In \cite[Lemma 4.17]{HY} it is proved that $\ov \Gr_{\bA^1}^{\la_1^\svee} \conv \hdots \conv \ov \Gr_{\bA^1}^{\la_k^\svee}$ is flat over $\bA^{k}$. From this, the required for us flatness is deduced in the same inductive way, as in the proof of \cite[Lemma 4.17]{HY} itself. We sketch it for the reader's convenience.
	
	As in \cite[(3.1.22)]{Zhu}, we define the $G_{\bA^1}(\cO)$-torsor $\bE$ over $\Gr_{\bA^1}^{\conv k} \times \bA^1$ as the moduli space of collections
	\[
	(c_1, \hdots, c_k, c_{k + 1}, \cP_1, \hdots \cP_k, \al_1, \hdots, \al_k, \beta),
	\]
	where $(c_1, \hdots, c_k, \cP_1, \hdots \cP_k, \al_1, \hdots, \al_k)$ is a point of $\Gr^{\conv k}_{\bA^1}$, $c_{k + 1} \in \bA^1$, and $\beta$ is a trivialization $\cP_{k}\vert_{ \widehat{c}_{k + 1}} \xrightarrow{\sim} \cP^{triv} \vert_{\widehat{c}_{k + 1}}$ ($\widehat{c}_{k + 1}$ means the formal completion of $\bA^1$ at $c_{k + 1}$).
	
	Then the straightforward generalization of \cite[Proposition 4.15]{HY} tells that 
	\[
	\ov \Gr_{\bA^1}^{\la_1^\svee} \conv \hdots \conv \ov \Gr_{\bA^1}^{\la_k^\svee} \conv \ov \Fl_{\bA^1}^w \simeq \bE \vert_{\ov \Gr_{\bA^1}^{\la_1^\svee} \conv \hdots \conv \ov \Gr_{\bA^1}^{\la_k^\svee} \times \bA^1 } \times^{G_{\bA^1}(\cO)} \ov \Fl_{\bA^1}^w.
	\]
	Now the proposition follows from the flatness of $\ov \Gr_{\bA^1}^{\la_1^\svee} \conv \hdots \conv \ov \Gr_{\bA^1}^{\la_k^\svee}$ over $\bA^k$ and the flatness of $\ov \Fl_{\bA^1}^w$ over $\bA^1$, which is clear.
\end{proof}

\subsection{Line bundles on the global convolution diagrams}
Recall the Beilinson--Drinfeld Grassmannian $\Gr_{\bA^i}$, which is the moduli space of collections
\[
\{ (c_1, \hdots, c_i), \cP, \alpha \},
\]
where $(c_1, \hdots, c_i) \in \bA^i$, $\cP$ is a $G$-torsor on $\bA^1$, and $\alpha$ is a trivialization $\alpha: \cP \vert_{\bA^1 \setminus \{ c_1, \hdots, c_i \}} \xrightarrow{\sim} \cP^{triv}\vert_{\bA^1 \setminus \{ c_1, \hdots, c_i \}}$ (see, e.g. \cite[Section 3.1]{Zhu}). Recall also the ind-scheme, which we call the Beilinson--Drinfeld affine flag variety $\Fl_{\bA^{i}}$, that is, the moduli space
\[
\{ (c_1, \hdots, c_i), \cP, \alpha, \epsilon \},
\]
where $\{ (c_1, \hdots, c_i), \cP, \alpha\}$ is a point of $\Gr_{\bA^i}$, and $\epsilon$ is a $B$-reduction of $\cP$ at $c_i$ (it coincides with $\bf{Gr}_{\{1, \hdots, i - 1\}}$ of \cite[Section 4.3]{AR1} again with the same difference that $c_{i} = 0$ there, while we ``allow it to move'', which does not affect any proofs about these ind-schemes, including well-definedness).

Now, the Beilinson--Drinfeld Grassmannian $\Gr_{\bA^i}$ has the determinant line bundle, which we, abusing notation, denote by $\cL$ (see \cite[Section 3.1]{Zhu}). The Beilinson--Drinfeld flag variety $\Fl_{\bA^i}$ enjoys the projection to $\Gr_{\bA^i}$ by forgetting $\epsilon$ (this might be seen as the global analog of \eqref{fibration of flags over gr}), and we denote the pullback of the determinant line bundle under this projection by $\cL_{\La_0}$. Moreover, the choice of $\epsilon$ determines the tautological $B$-torsor over $\Fl_{\bA^i}$, and one can induce the line bundle, corresponding to any $B$-character $-\la$, which we denote as $\cL_\la$ (the minus sign here is standard, as, for example, when defining line bundles on $G/B$). Thus, there is a line bundle $\cL_\La$ on $\Fl_{\bA^i}$, corresponding to any affine weight $\La$.

Now, there are global analogs of the partial convolution morphisms \eqref{partial convolution morphisms} (see \cite[Section~4.5]{AR1} with the same minor difference as described above, and we refer to this paper and references therein for details). 
For $1 \leq i \leq k$ the morphism $\boldsymbol{\theta}_i: \Gr_{\bA^1}^{\conv k} \conv \Fl_{\bA^1} \rightarrow \Gr_{\bA^i}$ is defined by:
\[
\{ (c_1, \hdots, c_{k + 1}), \mathcal{P}_1, \hdots, \cP_{k + 1}, \alpha_1, \hdots,  \alpha_{k + 1}, \epsilon \} \mapsto \{(c_1, \hdots, c_i), \cP_i, \al_1 \circ \hdots \circ \al_i \}.
\]
The morphism $\boldsymbol{\theta}_{k + 1}: \Gr_{\bA^1}^{\conv k} \conv \Fl_{\bA^1} \rightarrow \Fl_{\bA^{k + 1}}$ is defined similarly:
\[
\{ (c_1, \hdots, c_{k + 1}), \mathcal{P}_1, \hdots, \cP_{k + 1}, \alpha_1, \hdots,  \alpha_{k + 1}, \epsilon \} \mapsto \{(c_1, \hdots, c_{k + 1}), \cP_{k + 1}, \al_1 \circ \hdots \circ \al_{k + 1}, \epsilon \}.
\] 

Using that, for $(\ell_1, \hdots, \ell_k, \La) \in \bZ^k \oplus \widehat{P}$ we define the line bundle on $\Gr_{\bA^1}^{\conv k} \conv \Fl_{\bA^1}$:
\[
\cL^{(\ell_1, \hdots, \ell_k, \La)} := (\boldsymbol{ \theta}_1, \hdots, \boldsymbol{\theta}_{k + 1})^* ( \cL^{\ell_1 - \ell_2} \boxtimes \hdots \boxtimes \cL^{\ell_k - \ell} \boxtimes \cL_\La),
\]
where $\ell$ is the level of $\La$.

Recall \eqref{fibers of global convolution diagram}. It is evident from definition that over the point $(c_1, \hdots, c_{k + 1})$ with $c_i \neq c_j$, the map $\boldsymbol{\theta}_{i}$ is the projection to the first $i$ factors, while over 0 it coincides with the morphism $\theta_i$. Thus, we have:

\begin{equation} \label{fibers of line bundle L}
\begin{aligned} 
\cL^{(\ell_1, \hdots, \ell_k, \La)} \vert_\bc &= \cL^{\ell_1} \boxtimes \hdots \boxtimes  \cL^{\ell_{k}} \boxtimes \cL^\La, \text{ if } \bc = (c_1, \hdots c_{k + 1}), c_i \neq c_j; \\
\cL^{(\ell_1, \hdots, \ell_k, \La)} \vert_{0} &= \cL^{(\ell_1, \hdots, \ell_k, \La)}. \
\end{aligned}
\end{equation}

\section{Main results} \label{section main result}


Let $M_i$ be $\g[t]$-modules as in Subsection \ref{FL fusion profuct}.
We prove the following key theorem, which tells how to compute $M_1 \ast \hdots \ast M_k$ if these modules appear in a certain geometric context.

\begin{thm} \label{geometric fusion}
	Let $f: X \rightarrow \bA^k$ be a proper morphism of schemes. Suppose the group scheme $G^{sc}_{\bA^k}(\cO) \rtimes \bC^\times$ acts on $X$, in such a way that $f$ is $\bC^\times$-equivariant (here $\bC^\times$ acts by the loop rotation on $G^{sc}_{\bA^k}(\cO)$, and by the standard dilations on $\bA^k$).
	
	Suppose $\cF$ is a $G^{sc}_{\bA^k}(\cO) \rtimes \bC^\times$-equivariant coherent sheaf on $X$, flat over $\bA^k$, and we have:
	\begin{enumerate}
		\item For any $\bc = (c_1, \hdots, c_k) \in \bA^k$ with pairwise distinct coordinates, $H^0(X\vert_\bc, \cF\vert_\bc)^*$ is the $G^{sc}_{\bA^k}(\cO)\vert_\bc = G^{sc}(\cO)^{\times k}$-module isomorphic to $M_1(c_1) \boxtimes \hdots \boxtimes M_k(c_k)$.
		\item The dual zero fiber $M = H^0(X\vert_0, \cF|_0)^*$ is a $\g[t]$-module with a unique up to scaling highest $\g$-weight vector, and this vector is cyclic.
		\item Higher cohomologies of all fibers vanish: $H^{>0}(X\vert_\bc, \cF\vert_\bc) = 0$ for all $\bc \in \bA^k$.
	\end{enumerate}
	
	Then the fusion product of modules $M_1, \hdots, M_k$ is independent of the choice of parameters and is isomorphic to the dual zero fiber of $H^0(X, \cF)$:
	\[
	M_1 \ast \hdots \ast M_k \simeq M.
	\]
	
	Moreover, it is associative, which means that for all $m$ one has:
	\[
	(M_1 \ast \hdots \ast M_m) \ast M_{m + 1} \ast \hdots \ast M_k \simeq M.
	\]
\end{thm}
\begin{proof}
	Due to \cite[Lemma 4.1]{DFF} we have a natural inclusion of $G^{sc}(\cO)$ to the group of global sections of $G^{sc}_{\bA^k}(\cO)$ (viewed as a sheaf on $\bA^k$). Hence $H^0(X, \cF)$ admits a natural $G^{sc}(\cO)$-action. It also admits the $\cO_{\bA^k}(\bA^k)\simeq \bC[z_1, \hdots,z_k]$-action, and these two actions commute. So we can regard $H^0(X, \cF)$ as a $U(\g[t])$-$\cO(\bA^k)$-bimodule.
	
	Moreover, due to $\bC^\times$-equivariance of $f$, there is a nonnegative grading on $H^0(X, \cF)$, which is respected both by $t$-degree grading on $U(\g[t])$ and by the standard grading on $\bC[z_1, \hdots,z_k] \simeq \cO_{\bA^k}(\bA^k)$.
	
	Due to the Euler characteristic invariance in flat families and higher cohomology of fibers of $\cF$ vanishing, $H^0(X, \cF)$ is flat and hence free over $\cO(\bA^k)$. Let us consider the $\cO(\bA^k)$-dual $U(\g[t])$-$\cO(\bA^k)$-bimodule $\Gamma = H^0(X, \cF)^\vee$. 
	
	For any (closed) point $\bc \in \bA^k$, let us denote by $\bC_\bc$ the one dimensional $\cO(\bA^k)$-module, obtained as a quotient by the maximal ideal, corresponding to $\bc$.
	
	By the assumption (ii), $\Gamma \T_{\cO(\bA^k)} \bC_0$ is $U(\g[t])$-generated by a highest weight vector. Let $v$ be its homogeneous lift to $\Gamma$. The set $U(\g[t])v$ spans the zero fiber $\Gamma \T_{\cO(\bA^k)} \bC_0$. Hence, by the graded Nakayama lemma, $\Gamma$ is generated by $v$ as a $U(\g[t])$-$\cO(\bA^k)$-bimodule.

	It follows that for any $\bc \in \bA^k$ the module $\Gamma \T_{\cO(\bA^k)} \bC_\bc$ is $U(\g[t])$-cyclic, as $v \T 1$ is the cyclic vector. As in the definition of the fusion product, we obtain a filtration $F$ on $\Gamma \T_{\cO(\bA^k)} \bC_\bc$, induced by the $t$-grading on $U(\g[t])$.
	
	Let $\bc = (c_1, \hdots, c_k) \in \bA^k$ be a point with pairwise distinct coordinates $c_i$. We now claim that one has a surjection of cyclic $U(\g[t])$-modules:
	\begin{equation} \label{surjection from zero fiber to graded}
	\Gamma \T_{\cO(\bA^k)} \bC_0 \twoheadrightarrow \gr_F (\Gamma \T_{\cO(\bA^k)} \bC_\bc).
	\end{equation}
	
	We need to show that if $u \in U(\g[t])$ annihilates the cyclic vector of the left-hand side module, then it also annihilates the cyclic vector of the right-hand side module. 
	We may assume $u$ is homogeneous. The fact that it annihilates the cyclic vector $v \otimes 1 \in \Gamma \T_{\cO(\bA^k)} \bC_0$ means that there are $u' \in U(\g[t])$ and $P \in \cO(\bA^k)$ such that $P(0) = 0$ and
	\begin{equation} \label{vanishing condition}
	(u - Pu')v = 0
	\end{equation}
	in $\Gamma$. We may assume that $P, u'$ and the relation \eqref{vanishing condition} are homogeneous, and hence $\deg (u') < \deg (u)$ in $U(\g[t])$. \eqref{vanishing condition} implies also that one has $(u - P(\bc)u')(v \T 1) = 0$ in $\Gamma \T_{\cO(\bA^k)} \bC_\bc$, and because of the degrees inequality it follows that $u$ annihilates the cyclic vector in 
    \[
     \gr (\Gamma \T_{\cO(\bA^k)}\bC_\bc).
    \]
 Thus, \eqref{surjection from zero fiber to graded} is indeed a surjection.
	
	However, since $\Gamma$ is free over $\cO({\bA^k})$, modules in \eqref{surjection from zero fiber to graded} have equal dimensions, and thus it is in fact an isomorphism. Hence, using the assumption (i), we obtain
	\[
	M \simeq \Gamma \T_{\cO(\bA^k)} \bC_0 \simeq \gr_F (\Gamma \T_{\cO(\bA^k)} \bC_\bc) \simeq \gr_F (M_1(c_1) \T \hdots \T M_k(c_k)).
	\]
	Note that the filtration here coincides with the one in the definition of fusion product, as both are induced by the $t$-degree grading on $U(\g[t])$, applied to the unique up to scalar highest weight vector.
	
	Hence, the first part of the Theorem is proved.
	
	The second part about the associativity (in the assumption that $\Gamma$ is free, which is true in our case), is proved in \cite[Propositin A.1(b)]{DFF}.
\end{proof}

\begin{rem}
	In fact one can also prove that the $\cO(\bA^k)$-dualized global sections $H^0(X, \cF)^\vee$ are isomorphic to the $U(\g[t])$-$\cO({\bA^k})$-sub-bimodule of $M_1[z_1] \T \hdots \T M_k[z_k]$, generated by the tensor product of cyclic vectors.
	
	For $X$ being a Schubert variety in Beilinson--Drinfeld Grassmannian and $\cF$ being a power of the determinant line bundle, it is the main result of \cite{DFF}, which is a Borel--Weil-type theorem for the Beilinson--Drinfeld Grassmannian. The resulting modules are called global Demazure modules. We refer to this paper for details.

    For $X = \ov \Gr_{\bA^1}^{\la_1^\svee} \conv \hdots \conv \ov \Gr_{\bA^1}^{\la_k^\svee} \conv \ov \Fl_{\bA^1}^w$ and $\cF = \cL^{(\ell_1, \hdots, \ell_k, \La)}$, this result would be the global analog of Proposition \ref{sections of convolution diagram}. The resulting modules can be called the generalized global Demazure modules.
\end{rem}

\begin{cor} \label{fusion of general demazures}
	Let $\la_1^\svee, \hdots, \la_k^\svee$ be dominant coweights, $\mu$ be a dominant weight and $\ell_1 \geq \hdots \geq \ell_k \geq \ell$ be nonnegative integers. Then
	\begin{multline*}
	D(\ell_1, \ell_1 \la_1^\svee) \ast \hdots \ast D(\ell_k, \ell_k \la_k^\svee) \ast D(\ell, \mu) \simeq \\
	\bold D \Bigl( \tau_{w_0 \la_1^\svee} (\ell_1 - \ell_2)\La_0, \tau_{w_0(\la_1^\svee + \la_2^\svee)} (\ell_2 - \ell_3)\La_0, \hdots, \tau_{w_0(\la_1^\svee + \hdots + \la_k^\svee)}(\ell_k - \ell) \La_0, \ell \La_0 + \ell w_0(\la_1^\svee + \hdots + \la_k^\svee) + w_0\mu  \Bigr),
	\end{multline*}
	and the fusion product of these modules is associative.
\end{cor}

Note that the module on the right-hand side is more natural from the geometric point of view, as it is just the dual space of sections of a line bundle on the convolution of Schubert varieties due to Proposition \ref{sections of convolution diagram}.

An alternative way to describe it algebraically is as the cyclic submodule generated by the tensor product of cyclic vectors inside
\begin{multline*}
D(\ell_1 - \ell_2, (\ell_1 - \ell_2) \la_1^\svee)  \T D(\ell_2 - \ell_3, (\ell_2 - \ell_3)(\la_1^\svee + \la_2^\svee)) \T \hdots \\
\T  D(\ell_{k} - \ell_{k + 1}, (\ell_{k} - \ell_{k + 1})(\la_1^\svee + \hdots + \la_{k}^\svee)) \T D(\ell, \ell \iota(\la_1^\svee + \hdots + \la_k^\svee) + \mu).
\end{multline*}

\begin{proof}
	Let $D(\ell, \mu) \simeq \bold D(w\La)$, $w \in \widehat{W}, \La \in \widehat{P}^+$.
	
	We apply Theorem \ref{geometric fusion} to the variety $X = \ov \Gr_{\bA^1}^{\la_1^\svee} \conv \hdots \conv \ov \Gr_{\bA^1}^{\la_k^\svee} \conv \ov \Fl_{\bA^1}^w$ and its line bundle $\cF = \cL^{(\ell_1, \hdots, \ell_{k}, \La)}$. Let us check that all assumptions of the theorem are satisfied. 
	
	The flatness assumption is verified in Proposition \ref{flatness of global convolution}.
	
	The assumption (i) is satisfied, as due to \eqref{fibers of convolution schubert variety}, \eqref{fibers of line bundle L}, \eqref{cohomology of spherical schubert variety} and \eqref{cohomology of schubert variety}, for $\bc$ with pairwise distinct coordinates we have
	\begin{multline*}
	H^0(\ov \Gr_{\bA^1}^{\la_1^\svee} \conv \hdots \conv \ov \Gr_{\bA^1}^{\la_k^\svee} \conv \ov \Fl_{\bA^1}^w \vert_{\bc}, \cL^{(\ell_1, \hdots, \ell_{k}, \La)}\vert_\bc)^* \simeq \\ 
	H^0(\ov \Gr^{\la_1^\svee} \times \hdots \times \ov \Gr^{\la_k^\svee} \times \ov \Fl^w, \cL^{\ell_1} \boxtimes \hdots \boxtimes \cL^{\ell_k} \boxtimes \cL^{\La})^* \simeq \\
	D(\ell_1, \ell_1 \la_1^\svee) \boxtimes \hdots \boxtimes D(\ell_k, \ell_k \la_k^\svee) \boxtimes D(\ell, \mu).
	\end{multline*}
	
	The assumption (ii) is satisfied due to \eqref{fibers of convolution schubert variety}, \eqref{fibers of line bundle L} and Proposition \ref{sections of convolution diagram}.
	
	We check the assumption (iii) for $\bc = 0$ and for $\bc$ with pairwise distinct coordinates, and for the rest it follows by the semi-continuity theorem. For $\bc = 0$ it follows by \eqref{fibers of convolution schubert variety}, \eqref{fibers of line bundle L} and Proposition \ref{sections of convolution diagram}, and for $\bc$ with pairwise distinct coordinates it follows from \eqref{fibers of convolution schubert variety}, \eqref{fibers of line bundle L}, \eqref{cohomology of spherical schubert variety} and \eqref{cohomology of schubert variety}.
\end{proof}

\begin{rem} \label{fusion of quotient}
	One can also deduce the following result from Theorem \ref{geometric fusion}.
	Using same notation as in Corollary \ref{fusion of general demazures}, let, in addition, for some $1 \leq m \leq k$, $\la_m^{\svee \prime}$ be a dominant coweight such that $\la_m^{\svee \prime} \leq \la_m^{\svee}$ (hence $D(\ell_m, \ell_m \la_m^{\svee \prime}) \subset D(\ell_m, \ell_m \la_m^{\svee})$), and $\mu^\prime$ be a dominant weight such that naturally $D(\ell, \mu^\prime) \subset D(\ell, \mu)$. Then:
	\begin{multline} \label{fusion of quotient formula 1}
	D(\ell_1, \ell_1 \la_1^\svee) \ast \hdots \ast \Bigl( D(\ell_m, \ell_m \la_m^\svee) / D(\ell_m, \ell_m \la_m^{\svee \prime}) \Bigr) \ast  \hdots \ast D(\ell_k, \ell_k \la_k^\svee) \ast D(\ell, \mu) \\
	\left. \simeq \Bigl( D(\ell_1, \ell_1 \la_1^\svee) \ast \hdots \ast D(\ell_k, \ell_k \la_k^\svee) \ast D(\ell, \mu) \Bigr) \middle/ \right. \\
	\Bigl( D(\ell_1, \ell_1 \la_1^\svee) \ast \hdots \ast D(\ell_m, \ell_m \la_m^{\svee \prime})  \ast  \hdots \ast D(\ell_k, \ell_k \la_k^\svee) \ast D(\ell, \mu) \Bigr),
	\end{multline}
	and
	\begin{multline} \label{fusion of quotient formula 2}
	D(\ell_1, \ell_1 \la_1^\svee) \ast \hdots \ast D(\ell_k, \ell_k \la_k^\svee) \ast \Bigl( D(\ell, \mu) / D(\ell, \mu^\prime) \Bigr) \\
	\left. \simeq \Bigl( D(\ell_1, \ell_1 \la_1^\svee) \ast \hdots \ast D(\ell_k, \ell_k \la_k^\svee) \ast D(\ell, \mu) \Bigr) \middle/ \right.
	\Bigl( D(\ell_1, \ell_1 \la_1^\svee) \ast \hdots \ast D(\ell_k, \ell_k \la_k^\svee) \ast D(\ell, \mu') \Bigr)
	\end{multline}
	(one can show that the modules in the right-hand sides are contained one in the other so their quotients are well-defined).
	
	To get \eqref{fusion of quotient formula 1}, one should apply Theorem \ref{geometric fusion} to the scheme $X = \ov \Gr_{\bA^1}^{\la_1^\svee} \conv \hdots \conv \ov \Gr_{\bA^1}^{\la_k^\svee} \conv \ov \Fl_{\bA^1}^w$ and the sheaf $\cF = \cI_{\la_1^{\svee}, \hdots, \la_m^{\svee \prime}, \hdots, \la_k^{\svee}, w} \T  \cL^{(\ell_1, \hdots, \ell_{k}, \La)}$, where $\cI_{\la_1^{\svee}, \hdots, \la_m^{\svee \prime}, \hdots, \la_k^{\svee}, w}$ is the ideal sheaf of the closed subscheme $\ov \Gr_{\bA^1}^{\la_1^\svee} \conv \hdots \conv \Gr_{\bA^1}^{\la_k^\svee \prime} \conv \hdots \conv  \ov \Gr_{\bA^1}^{\la_k^\svee} \conv \ov \Fl_{\bA^1}^w$. The formula \eqref{fusion of quotient formula 2} is obtained similarly.
	
	It would be interesting to investigate if the Theorem \ref{geometric fusion} applies to the product of ideal sheaves of the form $\cI_{\la_1^{\svee}, \hdots, \la_m^{\svee \prime}, \hdots, \la_k^{\svee}, w}$ for different $m$'s, as that would allow to find the fusion product of several modules of the form $D(\ell_m, \ell_m \la_m^\svee) / D(\ell_m, \ell_m \la_m^{\svee \prime})$. In particular, the evaluation representation $V(\theta)$ has this form, as $V(\theta) \simeq D(1, \theta) / D(1, 0)$ (here $\theta$ is the maximal root).
\end{rem}

\begin{rem} \label{relation to other fusion modules}
	In fact, Theorem \ref{geometric fusion} explains, which geometric objects stand behind some known results about fusion product (which are particular cases of Corollary \ref{fusion of general demazures}). Namely, they follow from Theorem \ref{geometric fusion} applied to certain varieties and line bundles on them:
	\begin{enumerate}
		\item The isomorphism 
		\[
		D(\ell, \ell \la^\svee_1) \ast \hdots \ast D(\ell, \ell \la^\svee_k) \simeq D(\ell, \ell (\la^\svee_1 + \hdots + \la^\svee_k)),
		\]
		first discovered in \cite[Corollary 5]{FoLi2}, is a result of applying Theorem \ref{geometric fusion} to the closure of a $G_{\bA^k}(\cO)$-orbit in the Beilinson--Drinfeld Grassmannian, which is a degeneration of $\Gr^{\times k}$ to $\Gr$.
		
		\item The isomorphism 
		\[
		D(\ell, \ell \la^\svee_1) \ast \hdots \ast D(\ell, \ell \la^\svee_k) \ast V(\mu) \simeq D(\ell, \ell (\la^\svee_1 + \hdots + \la^\svee_k) + \mu)
		\]
		(with restrictions on $\mu$), first discovered in \cite[Theorem 1]{Ven}, is a result of applying Theorem \ref{geometric fusion} to the closure of a $G_{\bA^{k + 1}}(\cO)$-orbit in the iterated version of the ind-scheme, constructed in \cite{Gai}, which is a degeneration of $\Gr^{\times k} \times G/B$ to $\Fl$.
		
		\item The isomorphism 
		\[
		D(\ell, \ell \la^\svee_1) \ast \hdots \ast D(\ell, \ell \la^\svee_k) \ast D(\ell, \mu) \simeq D(\ell, \ell (\la^\svee_1 + \hdots + \la^\svee_k) + \mu),
		\] first discovered in \cite[Theorem 1]{CSVW}, \cite[Theorem 1]{VV} is a result of applying Theorem \ref{geometric fusion} to the closure of a $G_{\bA^{k + 1}}(\cO)$-orbit in the ind-scheme, similar to the previous case, which is a degeneration of $\Gr^{\times k} \times \Fl$ to $\Fl$.
		
		\item The isomorphism
		\begin{multline*}
		D(\ell_1, \ell_1 \la_1^\svee) \ast \hdots \ast D(\ell_k, \ell_k \la_k^\svee) \\
		\simeq \bold D \Bigl( \tau_{w_0 \la_1^\svee} (\ell_1 - \ell_2)\La_0, \tau_{w_0(\la_1^\svee + \la_2^\svee)} (\ell_2 - \ell_3)\La_0, \hdots, \tau_{w_0(\la_1^\svee + \hdots + \la_k^\svee)}\ell_k \La_0 \Bigr)
		\end{multline*}
		for $\ell_1 \geq \hdots \geq \ell_k$, which in simply-laced case coincides with the result of \cite[Theorem A]{Nao4}, is a result of applying Theorem \ref{geometric fusion} to the closure of a $G_{\bA^k}(\cO)$-orbit in the global convolution diagram of affine Grassmannians, which is a degeneration of $\Gr^{\times k}$ to $\Gr^{\conv k}$.
	\end{enumerate}
\end{rem}

\begin{rem} \label{fusion of evaluations}
	In the spirit of the previous remark, it would be interesting to explain geometrically and possibly extend to new cases some other results about fusion product. For instance, it is proved in \cite{KL} that 
	\[
	V(\ell \la) \ast V(\ell \la) \simeq D(\ell, 2 \ell \la) /( \g \T  t^2 \bC[t] )D(\ell, 2  \ell  \la) 
	\]
	for $\la \in \iota (P^{\vee +})$ (in fact, the result in \cite{KL} is more general, here we treat the simplest case).
	We expect that this isomorphism should have a geometric meaning, but we do not know how to derive it. 
	The natural analog for the global module $\Gamma$ from the proof of Theorem \ref{geometric fusion} would be the space of sections of the line bundle $\cL^\ell$ on the Beilinson--Drinfeld Schubert variety $\ov \Gr^{\lach,  \lach}$, which are constant along the orbits of the group $G_{\bA^2}(\cO)_1$, which is a degeneration of $G(\cO)_1^{\times 2}$ to $G(\cO)_2$ (here $G(\cO)_s$ stands for the kernel of the projection $G(\cO) \twoheadrightarrow G(\bC[[t]] / t^s)$). If one proves that this space is free as a module over $\cO_{\bA^2}(\bA^2)$, the result of \cite{KL} follows.
\end{rem}

\section{Coherent Satake category and quantum loop group} \label{section coherent satake and quantum loop group}

Perverse coherent sheaves were introduced in \cite{Bez}, \cite{AB}. The category of equivariant perverse coherent sheaves on the affine Grassmannian was studied in \cite{BFM}, \cite{CW}, \cite{FF}. In this section, we explain how our results are related to the convolution of certain objects in this category, and how it provides a bridge between two categories: perverse coherent sheaves on the affine Grassmannian and finite-dimensional modules over the quantum loop group, giving a conceptual explanation of the same cluster algebras appearance in \cite{CW} and \cite{DK2}.
Using this, we construct short exact sequences of perverse coherent sheaves for all simply-laced types, whose existence was conjectured in \cite{CW}.

\subsection{Coherent Satake category} Let us recall some notation of \cite{CW} (see this paper for definitions). Denote by $\cP_{coh}^{G^{sc}(\cO)} (\Gr)$ the category of $G^{sc}(\cO)$-equivariant perverse coherent sheaves on the affine Grassmannian. We do not encounter loop-rotation equivarience here for simplicity, though all statements we are going to prove have the analogs for the category $\cP_{coh}^{G^{^{sc}}(\cO) \rtimes \bC^\times} (\Gr)$, with proofs being essentially the same ($\bC^\times$-twist of sheaves corresponds to grading shift of their global sections, which we are going to consider).

Simple objects in this category are parametrized by pairs $(\la^\svee, \mu) \in P^\vee \times P$, such that $\mu$ is dominant for the Levi factor of the parabolic subgroup of $G$, corresponding to $\la^\svee$. Namely, there is a $G(\cO)$-equivariant vector bundle $\cV_\mu$ on $\Gr^{\la^\svee}$, and one defines $\cP_{\la^\svee, \mu} = IC(\Gr^{\la^\svee}, \cV_\mu[\frac{1}{2} \dim \Gr^{\la^\svee}])$.

The case of major importance is when $\cV_\mu$ is a power of the determinant line bundle $\cL^{\ell}$, since the corresponding perverse sheaves are real simple and they are expected to serve as cluster variables in the Grothendieck ring. In this case one has $\cP_{\lach, \ell} := IC(\Gr^{\la^\svee}, \cL^{ \ell} [ \langle \rho, \la^\svee \rangle]) = \cO_{\ov \Gr^{\la^\svee}} \T \cL^{ \ell}[\langle \rho, \la^\svee \rangle]$ (recall that $\dim \Gr^{\la^\svee} = \langle 2 \rho, \la^\svee \rangle$).
We also set $\cP_{i, \ell} :=  \cP_{\om_i^\svee, \ell}$ for $i \in I$.

Note that $H^{> 0}(\ov \Gr^{\lach}, \cL^{ \ell}) = 0$, hence $R\Gamma(\cP_{\lach, \ell})$ is a $\g[t]$-module concentrated in a single cohomological degree $\langle \rho, \lach \rangle$. We consider such modules as objects of the (abelian) category of $\g[t]$-modules. Then we have the following result, relating the convolution $\star$ in $\cP_{coh}^{G(\cO)} (\Gr)$ and Feigin--Loktev fusion product $\ast$ (it appeared for $\g \simeq \fsl_2$ and $\la_1^\svee = \la_2^\svee = \omega^\svee$ in \cite[Proposition 8.5]{BFM}):
\begin{prop} \label{sections of convolution of simple perverse sheaves}
	For any $\la_1^\svee, \la_2^\svee \in P^{\vee +}$, $\ell_1 \geq \ell_2 \geq 0$, $R\Gamma(\cP_{\la_1^\svee, \ell_1} \star \cP_{\la_2^\svee, \ell_2})$ is concentrated in a single cohomological degree $ \langle \rho, \la^\svee_1 + \la^\svee_2 \rangle$, and one has an isomorphism of $\g[t]$-modules
	\[\bigl(  R\Gamma (\cP_{\la_1^\svee, \ell_1})[- \langle \rho, \la^\svee_1  \rangle] \bigr)^* \ast \bigl( R\Gamma(\cP_{\la_2^\svee, \ell_2})[- \langle \rho,  \la^\svee_2 \rangle] \bigr)^* \simeq \\
\bigl( R\Gamma(\cP_{\la_1^\svee, \ell_1} \star \cP_{\la_2^\svee, \ell_2}) [- \langle \rho, \la^\svee_1 + \la^\svee_2 \rangle] \bigr)^*.\]
	A similar isomorphism holds for any number of modules.
\end{prop}
\begin{proof}
	By the definition of operation $\star$ (see \cite{BFM}, \cite{CW}), 
	\[
	\cP_{\la_1^\svee, \ell_1} \star \cP_{\la_2^\svee, \ell_2} =  R\theta_* (\cL^{(\ell_1, \ell_2)}[\langle \rho, \la^\svee_1 + \la^\svee_2 \rangle]),
	\]
	where $\theta: \ov \Gr^{\la_1^\svee} \conv \ov \Gr^{\la_2^\svee} \rightarrow \ov \Gr^{\la_1^\svee + \la_2^\svee}$ is the convolution morphism. Now the Proposition follows from Proposition \ref{sections of convolution diagram} and Corollary \ref{fusion of general demazures}. 
	
	The case of arbitrary number of modules is proved similarly.
\end{proof}

\subsection{Quantum loop group}

From now on we assume that $\g$ is of simply-laced type. Consider the loop algebra $\bm L \g = \g[t, t^{-1}]$, and the corresponding quantum loop group $U_q(\bm L \g)$. Denote by $\Rep \ U_q(\bm L \g)$ the category of its finite-dimensional representations. 
We denote Kirillov--Reshetikhin modules by $KR_{i, \ell}$ (here $i \in I, \ell \in \bZ_{\ge 0}$; we omit an extra spectral parameter in $\bC^{\times}$, as it will not matter in what follows).
Their tensor products fit into short exact sequences, called Q-systems, see \cite[Theorem 6.1]{Her}. In \cite{Ked, DK2} the cluster meaning of Q-systems was noticed. In the cluster algebra described there, the role of cluster variables in the initial cluster seed is played by certain Kirillov--Reshetikhin modules (up to a sign). 

For certain modules $M \in \Rep \ U_q(\bm L \g)$ their \textit{graded limit} $M_{loc}$ is defined, which is a finite-dimensional graded $\g[t]$-module; we refer to \cite{BCKV} for a nice survey on this. \cite[Theorem A]{Nao4} claims that the following holds:
\begin{prop} \label{graded limit of product of KR}
	For any $i_1, i_2 \in I$, $\ell_1, \ell_2 \geq 0$, there is an isomorphism of $\g[t]$-modules:
	\[
	(KR_{i_1, \ell_1})_{loc} \ast (KR_{i_2, \ell_2})_{loc} \simeq (KR_{i_1, \ell_1} \T KR_{i_2, \ell_2})_{loc},
	\]
	and similar isomorphism holds for any number of modules.
\end{prop}

\subsection{Relation between the two via fusion product}
Combining the above results, we get the following
\begin{prop} \label{relation between perverse coherent and quantum group}
For any $i_1, \hdots, i_k \in I$, $\ell_1 \geq \hdots \geq \ell_k \geq 0$, there is an isomorphism of $\g[t]$-modules:
\[
(KR_{i_1, \ell_1} \T \hdots \T KR_{i_k, \ell_k})_{loc} \simeq \bigl( R\Gamma (\cP_{i_1, \ell_1} \star \hdots \star \cP_{i_k, \ell_k})[ - \frac{1}{2} \dim \Gr^{\om_{i_1}^\svee + \hdots + \om_{i_k}^\svee }] \bigr)^*.
\]
\end{prop}
\begin{proof}
Using Propositions \ref{sections of convolution of simple perverse sheaves} and \ref{graded limit of product of KR}, it is sufficient to prove this Proposition for $k = 1$. The case $k = 1$ follows from
\[
(KR_{i, \ell})_{loc} \simeq D(\ell, \ell \om_i^\svee) \simeq 
\bigl( R\Gamma(\cP_{i, \ell}) [-\frac{1}{2} \dim \Gr^{\om_i^\svee}] \bigr)^*,
\]
where the first isomorphism is \cite[Theorem 4]{FoLi2}, and the second isomorphism is \eqref{cohomology of spherical schubert variety}.
\end{proof}

We do not know if this proposition is a shadow of existence of some monoidal functor between subcategories of $\cP_{coh}^{G^{sc}(\cO)} (\Gr)$ and of $\Rep \ U_q(\bold L \g)$.

\subsection{Cluster relations in coherent Satake category} 
We now give an application of this connection. In \cite[Proposition 2.17]{CW}, the authors prove that certain perverse coherent sheaves fit into short exact sequences. These correspond to cluster relations in the Grothendieck ring. Their proof works for $G = GL_n$ only, as it uses the lattice realization of affine Grassmannian, though a similar result is expected for any type, see \cite[Conjecture 1.10]{CW}. We provide a proof in all simply-laced types. Thus, it can be seen as a first step towards the proof of \cite[Conjecture 1.10]{CW}.

First, we need the following
\begin{lem} \label{perverse sitting in a single cohomological degree}
	For any dominant coweights $\la^\svee_1, \ldots, \la^\svee_k$ and integers $\ell_1 \geq \hdots \geq \ell_k$ the perverse coherent sheaf $\cP_{\la_1^\svee, \ell_1} \star \hdots \star \cP_{\la^\svee_{k}, \ell_k}$ on $\Gr$ is concentrated in a single cohomological degree $\langle \rho, \la^\svee_1 + \hdots + \la^\svee_k \rangle$ (with respect to the standard t-structure).
\end{lem}
\begin{proof}
	By Serre vanishing, it is sufficient to check that for all $N \gg 0$ the complex 
 \[R\Gamma ((\cP_{\la_1^\svee, \ell_1} \star \hdots \star \cP_{\la^\svee_{k}, \ell_k}) \T \cL^{N})\]
 is concentrated in this cohomological degree. Let $\theta: \ov \Gr^{\la^\svee_1} \conv \hdots \conv \ov \Gr^{\la^\svee_k} \rightarrow \ov \Gr^{\la^\svee_1 + \hdots + \la^\svee_k}$ be the convolution morphism. Then we have
	\begin{multline} \label{derived sections of perverse tensored by determinant}
	R \Gamma ((\cP_{\la_1^\svee, \ell_1} \star \hdots \star \cP_{\la^\svee_{k}, \ell_k}) \T \cL^N ) \stackrel{(i)}{\simeq} R \Gamma ( (R \theta_* \cL^{(\ell_1, \hdots, \ell_k)}[\frac{1}{2}\dim \Gr^{\la^\svee_1 + \hdots + \la^\svee_k}]) \T \cL^{N}) \stackrel{(ii)}{\simeq} \\ 
	R\Gamma (R \theta_* ( \cL^{(\ell_1, \hdots, \ell_k)}[\frac{1}{2}\dim \Gr^{\la^\svee_1 + \hdots + \la^\svee_k}] \T \theta^* \cL^{N}) ) \simeq R\Gamma (\cL^{(\ell_1 + N, \hdots, \ell_k + N)}[\frac{1}{2}\dim \Gr^{\la^\svee_1 + \hdots + \la^\svee_k}]),
	\end{multline}
	and the last complex is concentrated in the cohomological degree $\langle \rho, \la^\svee_1 + \hdots + \la^\svee_k \rangle$ for any $N > - \ell_k$ by Proposition \ref{sections of convolution diagram}.  Here the isomorphism $(i)$ holds by definition of the operation $\star$, and the isomorphism $(ii)$ is the projection formula.
\end{proof}

Before we state our next theorem, we make the following observation, which will be crutial for its proof.

Recall the standard Serre equivalence between the category of coherent sheaves on a projective variety $Z$ and the quotient of the category of finitely generated graded modules over the homogeneous coordinate ring of $Z$ by the subcategory of finite-dimensional modules. Namely, if $L$ is a very ample line bundle on $Z$, the functor $\cF \mapsto \bigoplus_{n > 0} \Gamma(\cF \T L^n)$ induces the equivalence between $\coh(Z)$ and $\bigoplus_{n \geq 0} \Gamma( L^n )$-$\grmod_{f.g.} \left/ \bigoplus_{n \geq 0} \Gamma( L^n ) \right.$-$\grmod_{f.d.}$.

Now, assume an affine group scheme $H$ acts on $Z$ and $L$ is an equivariant very ample line bundle. This induces the action of $H$ on the ring $\bigoplus_{n \geq 0} \Gamma( L^n)$ by graded automorphisms (and conversely, such an action on this ring determines the action and an equivariant very ample line bundle on its $\Proj$). Given a coherent sheaf $\cF$ on $Z$, the data of equivariant structure on it, under the Serre equivalence, induces the structure of $H$-equivariant graded $\bigoplus_{n \geq 0} \Gamma( L^n)$-module on $\bigoplus_{n > 0} \Gamma(\cF \T L^n)$ (and conversely, given such a module, we get an equvariant coherent sheaf on $Z$). $H$-equivariant morphisms on both sides are $H$-invariants in spaces of all morphisms, hence the equivariant morphisms of sheaves correspond to equivariant morphisms of modules. Thus, there is an equivariant version of the Serre equivalence:
\begin{equation} \label{equivariant serre equivalence}
\begin{aligned}
\coh^H(Z) &\simeq \bigoplus_{n \geq 0} \Gamma( L^n )\text{-}\grmod^H_{f.g.} \left/ \bigoplus_{n \geq 0} \Gamma( L^n ) \right.\text{-}\grmod^H_{f.d.} \\
\cF &\mapsto \bigoplus_{n > 0} \Gamma(\cF \T L^n).
\end{aligned}
\end{equation}

Let us introduce some notation. We denote by $C = (C_{ij})_{i,j}$ the Cartan matrix of $\g$. 
Denote by $\al_i$ ($\al^\svee_i$) simple roots (coroots). For $i, j \in I$ write $i \sim j$ if $C_{ij} < 0$.

\begin{thm} \label{exact triples of perverse coherent}
	For any $i \in I$, $\ell \in \bZ$ there are exact sequences of perverse coherent sheaves on $\Gr$:
	\begin{gather*}
	0 \rightarrow \cP^\star_{i, \ell}  \rightarrow \cP_{i, \ell + 1} \star \cP_{i, \ell - 1} \rightarrow  \cP_{i, \ell} \star \cP_{i, \ell} \rightarrow 0; \\
	0 \rightarrow \cP_{i, \ell} \star \cP_{i, \ell}  \rightarrow \cP_{i, \ell - 1} \star  \cP_{i, \ell + 1}  \rightarrow \cP^\star_{i, \ell} \rightarrow 0,
	\end{gather*}
	where
	\[ 
	\cP_{i, \ell}^\star = \bigstarlim_{j \sim i} \cP_{j, \ell}^{\star (- C_{ij})},  
	\]
	and the product is taken in arbitrary order. 
\end{thm}

Note that we assumed that $\g$ is of simply-laced type, and hence $C_{ij} = -1$ whenever $j \sim i$, but we write it in this form just to be consistent with non-simply-laced case.

\begin{proof}
	The second exact sequence is obtained from the first one by applying the Grothendieck--Serre duality. We prove the first.
	
	Tensoring by a power of $\cL$ we may assume $\ell = 0$. 
	
	Note that
	\[
	\sum_{j \sim i} -C_{ij} \om^\svee_j = \left( \sum_{j \in I} -C_{ij} \om^\svee_j \right) + 2\om^\svee_i = - \alpha^\svee_i + 2 \om_i^\svee.
	\]
	Hence, 
	\begin{equation} \label{last term}
	\cP^\star_{i, 0} \simeq  \bigstarlim_{j \sim i} \cP^{\star - (C_{ij})}_{j, 0}   \simeq \cO_{\ov \Gr^{2 \om_i^\svee - \al_i^\svee}} [\langle \rho, 2\om_i^\svee - \al_i^\svee \rangle ]. 
	\end{equation}
	Also, 
	\begin{equation} \label{middle term}
	\cP_{i, 0} \star \cP_{i, 0} \simeq \cP_{2\om_i^\svee, 0} \simeq \cO_{\ov \Gr^{2 \om_i^\svee}} [\langle \rho, 2 \om_i^\svee \rangle].
	\end{equation}
	
	Now, it is sufficient to construct the distinguished triangle of the form
	\begin{equation} \label{twisted exact triple}
	\cP_{i, 1} \star \cP_{i, - 1} \rightarrow \cP_{i, 0} \star \cP_{i, 0} \rightarrow   \cP^\star_{i, 0} [1]
	\end{equation}
	
	By Lemma \ref{perverse sitting in a single cohomological degree}, all three terms of \eqref{twisted exact triple} lie in a single cohomological degree. Clearly, $\langle \rho, 2 \om_i^\svee \rangle = \langle \rho, {2 \om_i^\svee - \al_i^\svee} \rangle + 1$, so in fact they lie in the same cohomological degree.
	
	Hence, we are to prove that in the abelian category of coherent sheaves on $\Gr$ there is an exact triple
	\[
	0 \rightarrow \cP_{i, 1} \star \cP_{i, - 1} [-2 \langle \rho, \om_i^\svee \rangle] \stackrel{}{\rightarrow} \cP_{i, 0} \star \cP_{i, 0} [- 2 \langle \rho, \om_i^\svee \rangle] \stackrel{}{\rightarrow}   \cP^\star_{i, 0} [- 2 \langle \rho, \om_i^\svee \rangle + 1] \rightarrow 0,
	\]
	whose terms we denote by $\cA, \cB, \cC$ (so we need to construct the exact triple $\cA \stackrel{\phi}{\rightarrow} \cB \stackrel{\psi}{\rightarrow} \cC$).
	
	Using \eqref{middle term} and \eqref{last term}, we see the existence of a surjective morphism 
	\begin{equation} \label{map from B to C}
	\cB \simeq \cO_{\ov \Gr^{2 \om_i^\svee}} \stackrel{\psi}{\rightarrow} \cO_{\ov \Gr^{2 \om_i^\svee - \al_i^\svee}} \simeq \cC.
	\end{equation}

	To construct $\phi$ and prove the exactness, we consider these sheaves as sheaves on a finite-dimensional projective variety $\ov \Gr^{2\om_i^\svee}$. 
    In view of the equivalence \eqref{equivariant serre equivalence}, the functor 
    \[
    \cF \mapsto \bigoplus_{n > 0} \Gamma(\cF \T \cL^n)
    \]
    from $\coh^{G^{sc}(\cO)}(\ov \Gr^{2\om_i^\svee})$ to equivariant modules over the ring $\bigoplus_{n \geq 0} \Gamma( \cL^n )$ is fully faithful, and thus, it is sufficient to construct the short exact triple of $\bigoplus_{n \geq 0} \Gamma( \cL^n )$-modules.
	
	Due to \eqref{derived sections of perverse tensored by determinant} together with Corollary \ref{fusion of general demazures}, for $n > 0$ we have:
	\begin{align} \label{sections of sheaves A, B, C}
	\begin{split}
	\Gamma (\cA \T \cL^n)^* &\simeq D(n + 1,(n + 1) \om_i^\svee) \ast D(n - 1, (n - 1) \om_i^\svee); \\
	\Gamma (\cB \T \cL^n)^* &\simeq D(n, n \om_i^\svee) \ast D(n, n \om_i^\svee); \\
	\Gamma (\cC \T \cL^n)^* &\simeq \bigastlim_{j \sim i} D(n, n \om_j^\svee)^{\ast (-C_{ij})}.
	\end{split}
	\end{align}
	Note that these modules are nothing else but the graded limits of modules, appearing in Q-systems for $U_q(\bold L \g)$-modules. Using this, in \cite{CV} the authors deduced that these $\g[t]$-modules fit into exact triples (see Theorem 4 and Section 5.9; note that for classical types usage of quantum loop groups may be avoided in the proof).
	The map $\Gamma(\cC \otimes \cL^n)^* \hookrightarrow \Gamma(\cB \otimes \cL^n)^*$, constructed in \cite{CV} is the natural embedding $D(n, n(2\om_i^\svee - \al_i^\svee)) \hookrightarrow D(n, n2 \om_i^\svee)$; and the map $\Gamma(\cB \otimes \cL^n)^* \twoheadrightarrow \Gamma(\cA \otimes \cL^n)^*$ is the surjective map, which maps the natural cyclic vector to the natural cyclic vector.
	
	Hence, there is a short exact sequence of graded vector spaces
	\[
	0 \rightarrow \bigoplus_{n > 0} \Gamma (\cA \T \cL^n) \stackrel{\tilde \phi}{\rightarrow} \bigoplus_{n > 0} \Gamma (\cB \T \cL^n) \stackrel{\tilde \psi}{\rightarrow} \bigoplus_{n > 0} \Gamma (\cC \T \cL^n) \rightarrow 0.
	\]
	
	
	It is left to show that these maps are in fact morphisms of $\bigoplus_{n \geq 0} \Gamma(\cL^n)$-modules. 
	
	First, we claim that the morphism $\tilde \psi$, constructed in this way, coincides with the one, induced from $\psi$ of \eqref{map from B to C}. Indeed, the map constructed in \cite{CV} is induced by the embedding $D(n, n (2 \om_i^\svee - \alpha_i^\svee)) \hookrightarrow D(n, 2 n \om_i^\svee)$, whose geometric counterpart, the embedding $\ov \Gr^{2 \om_i^\svee - \alpha_i^\svee} \hookrightarrow \ov \Gr^{2 \om_i^\svee}$, was used to construct $\psi$. Since $\psi$ is a morphism of sheaves, $\tilde{\psi}$ is a morphism of graded modules over the ring $\bigoplus_{n \geq 0} \Gamma(\cL^n)$, so the claim is proved.
	
	Now we prove that $\tilde \phi$ is also a map of modules. First, note that 
    \[
    \bigoplus_{n \geq 0} \Gamma(\cL^n) = \bigoplus_{n \geq 0} D(n, 2n \om_i^\svee)^*,
    \] 
    and the multiplication is dual to the inclusion as a cyclic component: 
    \[
    D(n_1 + n_2, (n_1 + n_2)2 \om_i^\svee) \hookrightarrow D(n_1, 2n_1 \om_i^\svee) \T D(n_2, 2n_2 \om_i^\svee).
    \]
	$\cB \simeq \cO_{\ov \Gr^{2 \om_i^\svee}}$, so $\bigoplus_{n > 0} \Gamma (\cB \T \cL^n)$ is a regular $\bigoplus_{n \geq 0} \Gamma(\cL^n)$-module.
	
	Now, let us note that $\Gamma(\cL^{n_1})^*$ is a cyclic $\g[t]$-module, and due to \eqref{sections of sheaves A, B, C}, $\Gamma(\cA \T \cL^{n_2})^*$ is a tensor product of cyclic $\g[t]$-modules (for all $n_1, n_2 > 0$). Note that the natural cyclic vector of $\Gamma(\cL^{n_1})^*$ (defined up to scalar) is a linear functional on $\Gamma(\cL^{n_1})$, which sends a section to its value on the fiber $\cL^{n_1}\vert_{t^{ 2 \om_i^\svee}}$ ($t^{ 2 \om_i^\svee}$ is a point of $\ov \Gr^{2 \om_i^\svee}$). In the same way, the tensor product of natural cyclic vectors of $\Gamma(\cA \T \cL^{n_2})^*$ is the functional (again defined up to scalar), sending a section to its value at the same point $t^{2 \om_i^\svee}$. It follows that the dual $\Gamma(\cA \T \cL^{n_1 + n_2})^* \rightarrow \Gamma(\cL^{n_1})^* \T \Gamma(\cA \T \cL^{n_2})^*$ to the action map $\Gamma(\cL^{n_1}) \T \Gamma(\cA \T \cL^{n_2}) \rightarrow \Gamma(\cA \T \cL^{n_1 + n_2})$ maps the tensor product of cyclic vectors to the tensor product of cyclic vectors.

	In view of the above, the diagram
	\begin{equation} \label{multiplication commutes with morphism}
	\begin{tikzcd} 
	{\Gamma(\mathcal B \otimes \mathcal L^{n_1 + n_2})^*} & {\Gamma(\mathcal A \otimes \mathcal L^{n_1 + n_2})^*} \\
	{\Gamma(\mathcal L^{n_1})^* \otimes \Gamma(\mathcal B \otimes \mathcal L^{ n_2})^*} & {\Gamma(\mathcal L^{n_1})^* \otimes \Gamma(\mathcal A \otimes \mathcal L^{ n_2})^*}
	\arrow["{\tilde \phi_{n_1 + n_2}^*}", from=1-1, to=1-2]
	\arrow["{a_{\mathcal B}^*}"', from=1-1, to=2-1]
	\arrow["{\mathrm{id} \otimes \tilde \phi_{n_2}^*}"', from=2-1, to=2-2]
	\arrow["{a_{\mathcal A}^*}", from=1-2, to=2-2]
	\end{tikzcd}
	\end{equation}
	of $\g[t]$-modules is commutative for any $n_1, n_2 > 0$, since $\Gamma(\mathcal B \otimes \mathcal L^{n_1 + n_2})^*$ is $\g[t]$-cyclic and both $m_\cA^* \circ \tilde \phi_{n_1 + n_2}^*$ and $\id \T \tilde \phi_{n_2}^* \circ m_{\cB}^*$ map the cyclic vector of $\Gamma(\mathcal B \otimes \mathcal L^{n_1 + n_2})^*$ to the tensor product of cyclic vectors in $\Gamma(\mathcal L^{n_1})^* \otimes \Gamma(\mathcal A \otimes \mathcal L^{ n_2})^*$ (note again that by their construction in \cite{CV}, $\tilde{\phi}^*_n$ map cyclic vector to cyclic).
	
	Dualizing all the modules (and morphisms) in \eqref{multiplication commutes with morphism}, we obtain that $\tilde \phi$ is indeed a morphism of $\bigoplus_{n \geq 0} \Gamma(\cL^n)$-modules.
\end{proof}

We hope that a similar proof should exist in non-simply laced cases.

\begin{rem}
Note that in \cite[Lemma 3.1]{Ked} and \cite[Lemma 2.1]{DK2}, the classes of Kirillov--Reshetikhin modules are multiplied by roots of 1, so that a sign changes in the Q-systems relation, and after this change it becomes a cluster relation. In our case, this change of sign naturally appears in the proof from the cohomological shift \eqref{twisted exact triple}.

Moreover, it is stated in \cite{Ked} that the variables $Q_{i, \ell}$ have no meaning for $\ell < 0$ from the point of view of $U_q(\bold L \g)$ representation theory, but they are needed in order to get a cluster structure. In our case, the sheaves $\cP_{i, \ell}$ are well-defined and natural equally for any integer~$\ell$.

Both of these observations demonstrate that perhaps the category $\cP^{G^{sc}(\cO)}_{coh}(\Gr)$ is more natural for categorical realization of cluster algebra of \cite{Ked, DK2} than some category related to $\Rep U_q(\bold L \g)$, though these realizations are related, as our proof shows.
\end{rem}

\end{document}